\documentclass[pdflatex,sn-mathphys-num]{sn-jnl}

\usepackage{graphicx}%
\usepackage{multirow}%
\usepackage{amsmath,amssymb,amsfonts}%
\usepackage{amsthm}%
\usepackage{mathrsfs}%
\usepackage[title]{appendix}%
\usepackage{textcomp}%
\usepackage{manyfoot}%
\usepackage{booktabs}%
\usepackage{algorithm}%
\usepackage{algorithmicx}%
\usepackage{algpseudocode}%
\usepackage{listings}%
\usepackage{CJKutf8}
\usepackage{hyperref}

\hypersetup{
    colorlinks=true,
    linkcolor=blue,
    filecolor=magenta,      
    urlcolor=cyan,
    pdftitle={Overleaf Example},
    pdfpagemode=FullScreen,
    }
    
\usepackage{xcolor,colortbl}

\definecolor{LightCyan}{rgb}{0.18, 0.52, 0.235}

\newcolumntype{a}{>{\columncolor{LightCyan}}c}

\usepackage{lipsum}

\newtheorem{theorem}{Theorem}
\newtheorem{lemma}{Lemma}
\newtheorem{definition}{Definition}
\newtheorem{corollary}{Corollary}
\newtheorem{conjecture}{Conjecture}

\theoremstyle{thmstyleone}%
\newtheorem{proposition}[theorem]{Proposition}%

\theoremstyle{thmstyletwo}%

\theoremstyle{thmstylethree}%

\newcommand{\BZ}{\mathbb{Z}}

\newcommand{\BC}{\mathbb{C}}
\newcommand{\BN}{\mathbb{N}}

\begin{document}

\title[Partition Functions and Kurepa Decomposition I]{Partition Functions and Kurepa Decomposition I: Algebraic computation and some physical Applications}


\author[1,2]{\fnm{Francis Atta} \sur{Howard}}\email{hfrancisatta@ymail.com; hfrancisatta@gmail.com}



\affil[1]{\orgdiv{International Chair in Mathematical Physics and Applications(CIPMA-UNESCO)}, \orgname{University of Abomey-Calavi}, \city{Cotonou}, \postcode{072 B.P. 50}, \country{Benin Republic}}

\affil*[2]{\orgdiv{NUCB(NIC)}, \orgname{Nagoya University of Business and Commence}, \orgaddress{\street{Sagamine}, \city{Nisshin-shi}, \postcode{470-0193}, \state{Aichi}, \country{Japan}}}



\abstract{
This paper examines the algebraic features of notable polynomial functions and explores their combinatorial aspects by presenting precise decompositions in terms of Dobinski numbers, Bell numbers, and moments generating functions. Additionally, a new equivalence to the Kurepa factorial is developed to help investigate the Kurepa conjecture.
In conclusion, we examine several physical phenomena related to Kurepa factorials, occupation number, Fermi-Dirac and Bose-Einstein distributions while exploring their algebraic characteristics.\\

\begin{CJK}{UTF8}{min}
本論文では、注目すべき多項式関数の代数的特徴を考察し、Dobinski数、Bell数、およびモーメント生成関数を用いた正確な分解を通して、その組合せ論的側面を探求する。さらに、Kurepa因数との新たな同値性を導出し、Kurepa予想の調査に役立てる。結論として、Kurepa階乗、占有数、フェルミ・ディラックおよびボース・アインシュタイン分布関数に関連するいくつかの物理現象を、その代数的特性を探求しながら検討する。
\end{CJK}
}

\keywords{Bell numbers, Partitions functions, Kurepa conjecture, Fermi-Dirac, Binary GCD algorithm, Normal ordering, Occupation number, Dobinski numbers}



\maketitle
\tableofcontents
\section{Introduction}\label{sec1}
A partition $\mathbb{P}$ of a set $X$ is a collection of nonempty, mutually exclusive subsets of $X$, termed blocks, whose union constitutes $X$. The Bell number ($\mathbf{Bell_n}$) represents the total number of partitions of the set $[n]$ or any other set containing $n$ elements. The Stirling number of the second kind, $S(n,k)$, represents the number of partitions of the set $[n]$ into exactly $k$ blocks; that is, $ \mathbf{Bell_n}=\sum_{k\geq 1}^nS(n,k).$ Several problems connect all these well-known numbers, authors such as Andrews George, Stefan De Wannemacker, Anne Gertsch, Don Zagier, R. J. Clarke, M. Klazar, C. Mijajlovic, Z.W. Sun and many more have made significant contributions to these areas \cite{andrews1998theory, sun2011curious, gertsch1996some, clarke1993derangements, guy2004unsolved, klazar2003bell, ivic2003kurepa, mijajlovic2021fifty, mijajlovic1990some, petojevic2023new}, specifically congruences representing Bell numbers and derangement numbers in terms of one another modulo any prime were derived by Don Zagier and Sun \cite{sun2011curious}. Anne further studied these congruences in her thesis \cite{gertsch1999congruences}, that is,  
\begin{equation*}
    K_p\equiv \sum_{0\leq k\leq p-1}\mathbf{Bell_{k}} \equiv \mathbf{Bell_{p-1}}- 1\mod{p}
\end{equation*}
was $\mathbf{Bell_{p-1}}\equiv \mathbf{Der_{p-1}}+ 1(\mod{p})$ and $K_p$ is the Kurepa factorial for prime $p$. In 1971, Duro Kurepa \cite{kurepa1971left} posed a question whether for every natural number $n\geq2$ the $G_n=\gcd(!n,n!)=2$, this problem remains open, and many researchers actively work on it \cite{barsky2004nombres, barsky2011erratum, mevstrovic2013variations, mijajlovic1990some, mijajlovic2021fifty, petojevic1999trees, petojevic2023new, andrejic2016searching, vsami1974m, ivic2003kurepa, abramov2022kurepa}.
Clarke \cite{clarke1993derangements} also worked on derangement numbers and showed that
\begin{align*}
\mathbf{Der_n}&= \frac{n!}{e}= n! \left( 1-1 + \frac{1^2}{2!} - \frac{1^3}{3!} + \cdots +\frac{(-1)^i}{i!} \right),\\
\mathbf{Der_n}&= n!\sum_{i=0}^n\frac{(-1)^i}{i!} = n!e^{-1}.
\end{align*}
Kellner \cite{kellner2004some} developed a relationship between the subfactorial function and Kurepa's left factorial function $K_n$. He outlined fundamental characteristics and congruences of both functions and provided a computed distribution of primes below 10,000 of $K_n,$ where
\begin{equation}
   K_n=\mathbf{Kurepa} = \mathbf{!n}=\sum_{i=0}^{n-1}i!   \quad \mbox{for all} \quad (i<n).
\end{equation}
Fabiano et al. \cite{unknown} in 2022 introduced a new description of  Kurepa’s conjecture and the relation to Bezout’s parameters and the Diophantine equation. They gave a numerical analysis that supports Kurepa’s hypothesis and the conjecture about distribution for Kurepa’s function.
It is quite surprising that Euclid's longstanding patriarchal technique is not the most efficient approach for ascertaining the greatest common divisor. In 1961, Josef Stein \cite{stein1967computational} devised a distinct $\gcd$ method primarily suited for binary arithmetic. This novel approach exclusively employs subtraction, parity checking, and the halving of even integers, eliminating the need for a division instruction.
Vladica Andrejić and Miloš Tatarević in 2016 \cite{andrejic2016searching} sought for a counterexample to the Kurepa hypothesis for any $p < 2^{34}$. They presented novel optimization approaches, executed computations with graphics processing units, and ultimately proposed a generalized version of Kurepa’s left factorial. 
Motivated by the above, I ask the following questions:
\begin{itemize}\label{Solve}
\item[(i)] What is the sum of Bell numbers ($\mathbf{Bell}_n$) ?
\item[(ii)] What is the sum of complementary Bell numbers ($\mathbf{invBell}_n$) ?
\item[(iii)] What is the logarithm of the Kurepa factorial?
\item[(iv)] Is there a trivial equivalence to the Kurepa conjecture?
\item[(v)] What are some possible physical applications of Kurepa?
\end{itemize}




\section{Preliminaries and Notation}\label{Prelim}
In this section, we shall define some notation and symbols of well-known numbers that will be used in the sequel. Note that notations are already used in this field of study; however, to avoid ambiguity in the subsequent development of concepts, we shall stick to the notation and symbols used in this paper; see \cite{bell1934exponential,bell1938iterated,touchard1933proprietes}

\subsection{Dobinski numbers, Bell numbers and Touchard polynomials}

\begin{definition}\label{Dobi}\cite{weisstein2002dobinski}
    Let $k$ be a nonnegative integer; then for all values of $n\geq 0$, we call 
    $$\mathbf{Dob}_n=\sum_{k=0}^{\infty}\dfrac{k^n}{k!}=\sum_{k=0}^{\infty}\dfrac{kk^{n-1}}{k(k-1)!}=\sum_{k=0}^{\infty}\dfrac{k^{n-1}}{(k-1)!},$$
    the Dobinski exponential series.
\end{definition}
Below are a few examples of the Dobinski exponentials for $n\geq0$:

\begin{align*}
    \exp(1)&= 1+\dfrac{2}{2!}+\dfrac{3}{3!}+\dfrac{4}{4!}+\cdots = 1+1+\dfrac{1}{1\cdotp 2}+\dfrac{1}{1\cdotp 2 \cdotp 3}+\cdots=\sum_{k=0}^{\infty}\dfrac{k}{k!}=\sum_{k=0}^{\infty}\dfrac{1}{(k-1)!}=\mathbf{Dob}_1\\ 
    2 \exp(1)&= 1+\dfrac{2^2}{2!}+\dfrac{3^2}{3!}+\dfrac{4^2}{4!}+\cdots = 1+2+\dfrac{3}{1\cdotp 2}+\dfrac{4}{1\cdotp 2 \cdotp 3}+\cdots=\sum_{k=0}^{\infty}\dfrac{k^2}{k!}=\sum_{k=0}^{\infty}\dfrac{k}{(k-1)!}=\mathbf{Dob}_2\\ 
    5 \exp(1)&= 1+\dfrac{2^3}{2!}+\dfrac{3^3}{3!}+\dfrac{4^3}{4!}+\cdots = 1+4+\dfrac{9}{1\cdotp 2}+\dfrac{16}{1\cdotp 2 \cdotp 3}+\cdots=\sum_{k=0}^{\infty}\dfrac{k^3}{k!}=\sum_{k=0}^{\infty}\dfrac{k^2}{(k-1)!}=\mathbf{Dob}_3\\ 
    15 \exp(1)&= 1+\dfrac{2^4}{2!}+\dfrac{3^4}{3!}+\dfrac{4^4}{4!}+\cdots = 1+8+\dfrac{27}{1\cdotp 2}+\dfrac{64}{1\cdotp 2 \cdotp 3}+\cdots=\sum_{k=0}^{\infty}\dfrac{k^4}{k!}=\sum_{k=0}^{\infty}\dfrac{k^3}{(k-1)!}=\mathbf{Dob}_4\\
    52 \exp(1)&= 1+\dfrac{2^5}{2!}+\dfrac{3^5}{3!}+\dfrac{4^5}{4!}+\cdots = 1+16+\dfrac{81}{1\cdotp 2}+\dfrac{256}{1\cdotp 2 \cdotp 3}+\cdots=\sum_{k=0}^{\infty}\dfrac{k^5}{k!}=\sum_{k=0}^{\infty}\dfrac{k^4}{(k-1)!}=\mathbf{Dob}_5\\ 
    203 \exp(1)&= 1+\dfrac{2^6}{2!}+\dfrac{3^6}{3!}+\dfrac{4^6}{4!}+\cdots = 1+32+\dfrac{243}{1\cdotp 2}+\dfrac{1024}{1\cdotp 2 \cdotp 3}+\cdots=\sum_{k=0}^{\infty}\dfrac{k^6}{k!}=\sum_{k=0}^{\infty}\dfrac{k^5}{(k-1)!}=\mathbf{Dob}_6\\
    877 \exp(1)&= 1+\dfrac{2^7}{2!}+\dfrac{3^7}{3!}+\dfrac{4^7}{4!}+\cdots = 1+64+\dfrac{729}{1\cdotp 2}+\dfrac{4098}{1\cdotp 2 \cdotp 3}+\cdots=\sum_{k=0}^{\infty}\dfrac{k^7}{k!}=\sum_{k=0}^{\infty}\dfrac{k^6}{(k-1)!}=\mathbf{Dob}_7 \\
      &\vdots \quad \quad \quad \quad \vdots \quad \quad \quad\vdots \quad \quad \quad \quad \quad \quad \quad\vdots \quad \quad \quad \quad \vdots \quad \quad \quad\vdots \quad \quad \quad \quad \quad \quad \quad\\
    \mathbf{Bell_{n}} \exp(1)&= 1+\dfrac{2^n}{2!}+\dfrac{3^n}{3!}+\dfrac{4^n}{4!}+\cdots = 1+2^n+\dfrac{3^n}{1\cdotp 2}+\dfrac{4^n}{1\cdotp 2 \cdotp 3}+\cdots=\sum_{n=1}^{\infty}\dfrac{k^n}{k!}=\sum_{k=1}^{\infty}\dfrac{k^{n-1}}{(k-1)!}=\mathbf{Dob}_{n}
\end{align*}

\begin{definition}\label{Bell}\cite{bell1934exponential, spivey2008generalized}
From the Dobinski's exponential series, we observe that 
\begin{align*}
     \mathbf{Bell_{n}} \exp(1)&= 1+\dfrac{2^n}{2!}+\dfrac{3^n}{3!}+\dfrac{4^n}{4!}+\cdots =\sum_{k=0}^{\infty}\dfrac{k^n}{k!}=\mathbf{Dob}_{n},
\end{align*}
the coefficients of the Dobinski series are the Bell numbers. The first-order Bell exponential series is given by 
$$\mathbf{Bell_{n}}=\dfrac{1}{\exp(1)}\sum_{k=0}^{\infty}\dfrac{k^n}{k!}=\dfrac{\mathbf{Dob}_n }{\exp(1)} $$
    for all $n\in \BN$.
\end{definition}
Next, we observe from Epstein's expansion \cite{epstein1939function} that
\begin{align*}
     e^{\exp{(ax)}}&=e^{\left ( 1+a_1 x + \dfrac{a_2 x^2}{2!}+\dfrac{a_3 x^3 }{3!}+\cdots \dfrac{a_n x^n}{n!}\right)}
\end{align*}
He then expressed 
$$a_n=\dfrac{1}{\exp}\sum_{k=0}^{\infty}\dfrac{k^n}{k!}=\dfrac{\mathbf{Dob}_{n}}{\exp}=\mathbf{Bell_{n}},$$
He also computed, $a_{-n}$, which lead to the Dirichlet series,
\begin{align*}
a_{-n}&=\dfrac{1}{\exp}\sum_{k=0}^{\infty}\dfrac{k^{-n}}{k!}=\dfrac{\mathbf{Dob}_{-n}}{\exp}\\ \nonumber
&=\dfrac{1}{\exp}\sum_{k=0}^{\infty}\dfrac{1}{k^n k!}=\dfrac{1}{\mathbf{Dob}_{n}\exp}
\\ \nonumber 
&=\mathbf{Bell_{-n}}.
\end{align*}
Using these relations, Epstein \cite{epstein1939function, williams1945numbers} derived for $n=1$
\begin{align*}
    \mathbf{Bell_{-1}}&=\dfrac{1}{\exp}\sum_{k=1}^{\infty}\dfrac{1}{k k!}\\ \nonumber
    &\dfrac{e^x}{x}=\dfrac{1}{x}+ \sum_{k=1}\dfrac{x^{k-1}}{k!}\\ \nonumber
    & \int_{0}^{1}\dfrac{e^x}{x} dx =\ln{x}+ 
    \sum_{k=1}\dfrac{x^{k-1}}{k!}
\end{align*}
if $x=1$ we have 
\begin{align*}
    \int_{0}^{1}\dfrac{e^x}{x}dx =\ln{1}+ 
    \sum_{k=1}\dfrac{1}{k\cdotp k!}
\end{align*}
where 
\begin{align*}
    \mathbf{Bell_{-1}}=\dfrac{1}{e}\int_{0}^{1}\dfrac{e^x}{x}dx 
\end{align*}
The well-known general form of the Dobinski function is given by:
\begin{align*}
    \dfrac{1}{\exp}\sum_{k=x}^{\infty}\dfrac{k^n}{(k-x)!}=\sum_{i=0}^k \left(\begin{array}{c} n \\ k \end{array}\right)\mathbf{Bell}_k \cdotp x^{n-k},
\end{align*}
we set $x=0$ we obtain the definition \ref{Bell}. Also, the Touchard polynomial($\mathbf{Tchd_{n}}(y)$) \cite{touchard1933proprietes} and its exponential generating function are given by;
\begin{equation}
    \mathbf{Tchd_{n}}(y)= \dfrac{1}{\exp{y}}\sum_{k=0}^{\infty}y^k \dfrac{k^n}{k!}
\end{equation}
and
\begin{equation}
   \mathbf{Tchd_{n}}(y)= \sum_{k=0}^{n} S(n,k)y^k
\end{equation}
with the exponential generating function $\sum_{n=0}^{\infty}\mathbf{Tchd_{n}}(y)\dfrac{y^n}{n!} e^{y((\exp x)-1)}$.
We set $y=1$ in the Touchard polynomial and obtain the Bell number.


\begin{definition}[Inverse Dobinski formula]\label{id}
    Let $k$ be a nonnegative integer; then for all values of $n\geq 0,$ we call 
    $$\mathbf{invDob}_n=\sum_{k=0}^{\infty}(-1)^k\dfrac{k^n}{k!},$$
    the inverse Dobinski number.
\end{definition}
Below are a few examples of the inverse Dobinski exponentials.
\begin{align*}
    \dfrac{-1}{e}&=\sum_{k=0}^{\infty}(-1)^k\dfrac{k}{k!}=\mathbf{invDob}_1\\ \nonumber
   \dfrac{0}{e}&= \sum_{k=0}^{\infty}(-1)^k\dfrac{k^2}{k!}=\mathbf{invDob}_2\\ \nonumber
    \dfrac{1}{e}&= \sum_{k=0}^{\infty}(-1)^k\dfrac{k^3}{k!}=\mathbf{invDob}_3\\ \nonumber
    \dfrac{1}{e}&= \sum_{k=0}^{\infty}(-1)^k\dfrac{k^4}{k!}=\mathbf{invDob}_4\\ \nonumber
    \dfrac{-2}{e}&=\sum_{k=0}^{\infty}(-1)^k\dfrac{k^5}{k!}=\mathbf{invDob}_5\\ \nonumber
   \dfrac{-9}{e}&=\sum_{k=0}^{\infty}(-1)^k\dfrac{k^6}{k!}=\mathbf{invDob}_6\\ \nonumber
    \dfrac{-9}{e}&= \sum_{k=0}^{\infty}(-1)^k\dfrac{k^7}{k!}=\mathbf{invDob}_7 \\ \nonumber
     &\vdots \quad \quad \quad \quad \vdots \quad \quad \quad\vdots \quad \quad \quad \quad \quad \quad \quad \\
    \dfrac{\mathbf{invBell_{n}}}{e}&= \sum_{k=1}^{\infty}(-1)^k\dfrac{k^n}{k!}=\mathbf{invDob}_{n}
\end{align*}

\begin{definition}\label{invBelo}\cite{amdeberhan2013complementary, rao1969numbers, wilf2005generatingfunctionology}
From definition \ref{id}, the inverse Bell (complementary Bell) numbers are given by:
\begin{align}
      \mathbf{invBell_{n}}=e\sum_{k=0}^{\infty}(-1)^k \dfrac{k^n}{k!}=\mathbf{invDob_{n}\cdotp e}.
\end{align}
\end{definition}

\subsection{Stirling number of the second kind}\label{sec}
Let \( S(n, k) \) denote the number of ways to partition a set of \( n \) elements into exactly \( k \) non-empty, unlabeled subsets. These satisfy the recurrence:
\[
S(n+1, k) = k \cdot S(n, k) + S(n, k-1),
\]
with conditions
$S(0, 0) =S(n, 1) =S(n, n) = 1, \quad S(n, 0) = 0 \text{ for } n > 0, \quad S(0, k) = 0 \text{ for } k > 0$ \cite{riordan2014introduction, comtet2012advanced}.
The exponential generating function for the Bell numbers is given by:
\[
 \sum_{n=0}^{\infty} \mathbf{Bell_n} \frac{x^n}{n!} =\sum_{k=0}^{\infty}\dfrac{(e^x - 1)^k}{k!}= e^{e^x - 1} \quad k\geq 0,
\]

\[
 \sum_{n=0}^{\infty} \mathbf{invBell_n} \frac{x^n}{n!} =\sum_{k=0}^{\infty}(-1)^k\dfrac{(e^x - 1)^k}{k!}= e^{1-e^x} \quad k\geq 0.
\]
The relation between Bell numbers and the complementary Bell numbers is as follows:
$\mathbf{invBell_n}=\sum_{k=1}^{n} (-1)^kS(n,k) $ and $\mathbf{Bell_n}=\sum_{k=1}^{n} S(n,k). $

\section{Kurepa Decompositions}\label{Kall}
According to Kurepa's hypothesis, $\gcd(!n,n!) = 2,\quad n > 1$. This is identical to demonstrating that $\gcd(p,!p) = 1$ for any odd primes $p$. According to Guy \cite{guy2004unsolved}, Mijajlovic has tested up to $p = 10^6$, Gallot also tested up to $2^{26}$ after that Jobling, Paul, have proceeded to $p < 144000000$, which is a little above, $p = 2^{27}$ with no instances of $\gcd(p,!p) > 1$  discovered. 
Milos Tatarevic searched till $10^9$, but could not find any counterexample in $2013$ \cite{andrejic2016searching}. In this section, we investigate well-known theorems about Kurepa factorials and Bell numbers, derangement (subfactorial) numbers, Stirling numbers of the second kind, and complementary Bell numbers and their connections with the Kurepa factorial \cite{clarke1993derangements, kurepa1971left, comtet2012advanced, riordan2014introduction, bell1938iterated, bell1934exponential, berend2010improved, amdeberhan2013complementary, kurepa1974some, ivic2003kurepa, fabianogcd}.

\subsection{Kurepa Factorial}

\begin{table}[!ht]
    \centering
    \begin{tabular}{|c|c|c|a|c|c|c|c|c|c|}\hline 
        $\mathbf{n}$& 0 & 1 & 2 & 3 &4  &5  & 6 & 7 &8 \\\hline 
         $\mathbf{n!}$& 1 & 1 & 2 & 6 & 24 & 120  & 720 & 5040 & 40320 \\\hline
        $\mathbf{!n}$ &  & 1 & 2 & 4 & 10 & 34 & 154 & 874 &5914 \\\hline
         $(-1)^n\mathbf{!n}$ &  & 1 & 0 & 2 & -4 & 20 & -100 & 620 &-4420 \\\hline
         $\mathbf{n!- !n}$ & 1 & 0 & 0 &  2& 14 & 86 & 566 &4166  &34406 \\\hline
          $\mathbf{!n+ n!}$ & 1 & 2 & 4 &  10& 34 & 154 & 874 & 5914  &46234 \\\hline
        $\dfrac{\mathbf{(!n- n!)}}{2}$ & 1/2 & 0 & 0 & 1 & 7 & 43 & 283 & 2083 & 17203\\\hline \rowcolor{LightCyan}
        $\mathbf{r}$ &  & 1/2 & 1 & 2 & 5 & 17 &77  &437  &2957 \\\hline  
         $\gcd(\mathbf{!n, n!})$& 1 & 1 & 2 & 2 & 2 &2  & 2 & 2 &2 \\\hline
        $\mathbf{Kurepa=2r}$ &  & 1 & 2 & 4 & 10 & 34 & 154 & 874 & 5914\\\hline
         $\mathbf{Der_n}$ & 1 &0  & 1 & 2 & 9 & 44 & 265 & 1854 &14833 \\\hline
         $\mathbf{Bell_n}$ &1 &1 & 2 & 5 & 15 & 52 & 203 & 877 & 4140 \\\hline
        $\mathbf{Dob_n}$& e  & 1e & 2e & 5e & 15e &52e  &203e  &877e  &4140e \\\hline
        $\mathbf{!ne}$ & !0e & !1e & !2e & !3e & !4e & !5e & !6e & !7e & !8e\\\hline
         $f(n)=\mathbf{invBell_n}$ & 1 & -1 & 0 & 1  & 1 & -2 & -9 & -9 & 50\\\hline
    \end{tabular}
    \caption{\cite{OEIS:A003422, OEIS:A000110, OEIS:A000587, OEIS:A000166}  Kurepa, Bell, Dobinski, Derangement and complementary Bell numbers}
    \label{table 1}
\end{table}

\begin{conjecture}\label{main}\cite{kurepa1971left, guy2004unsolved, mijajlovic2021fifty}
    For all, $n\geq2$, the common divisor between the left factorial $\mathbf{!n}$ and the right factorial $n!$ is $2$, that is, $$\gcd(\mathbf{!n}, n!)=2.$$

\end{conjecture}

\begin{lemma}\label{induct}
For integers $r\geq 2,$ the following consequences hold:
  \begin{enumerate}
        \item if $2$ divides $\mathbf{!n}+ n!$, then, there exists an $r$ such that $2\cdotp r= (\mathbf{!n}+n!);$
        \item  if $2| (\mathbf{!n}+n!)$ it immediately follows that $2| !(\mathbf{n}+1)$.
    \end{enumerate}
\end{lemma}
\begin{proof}
The basis of this proof is straightforward. For the proof of $(1)$, we observe from the table \ref{table 1} that
\begin{align*}
 \dfrac{(\mathbf{!n}- n!)}{2}+ \mathbf{!n}=r\\
n!- \mathbf{!n} + 2(\mathbf{!n}) =2r\\
n! + \mathbf{!n} =2r. 
\end{align*}
The second proof follows easily since the Kerupa factorial obeys the following recurrence $$ !(\mathbf{n}+1)=\mathbf{!n}+ n!$$ then $2| (\mathbf{!n}+n!)$ implies $2| !(\mathbf{n}+1)$, this finishes the proof.
\end{proof}

\begin{corollary}
    For all integers $r$ and $n$ the greatest common divisor
    \begin{align*}
      \gcd(\mathbf{!n},r)=r,\quad\mbox{and} \quad  
      \gcd\left(r,\frac{n!}{2}\right)=1
    \end{align*}
for all $n>2$.
\end{corollary}
\begin{proof}
    Details of this proof shall be discussed in subsequent sections.
\end{proof}

\begin{theorem}\label{Add}\cite{kurepa1971left}
 Consider the sequence $0!, 1!, 2!, 3!, 4!\cdots$, and the sum of any consecutive $n$ terms $$S_k(n)= k!+(k+1)!+ \cdots+(k+n-1)!,$$
    setting $k=0$ yields the famous Kurepa factorial
    $$S_0(n)= 0!+1!+ 2! +3!+ 4! +\cdots+(n-1)!=\mathbf{!n},$$ \\where $\mathbf{!n} = 0! +1! + 2! +\cdots +(n-1)! = \sum_{m=0}^{n-1} m!$.
    The product of the exponential series ($\exp{(x)}$) with the function $S_k(n)$ is given by:
\begin{align*}
S_k(n)e^x &= S_k(n)\sum_{n=0}\frac{x^n}{n!},
\end{align*}
if $k=0$ this function becomes 
\begin{align*}
S_0(n)e^x &= S_0(n)\sum_{n=0}\frac{x^n}{n!}=\mathbf{!n}e^x.
\end{align*}  
\end{theorem}

\begin{proof}
Let consider 
\begin{align*}
S_k(n) &= k! + (k+1)! + \cdots + (k+n-1)!\\\nonumber
S_0(n) &= 0!+1!+2!+3! + \cdots + (n-1)!\\ \nonumber
S_1(n) &= 1!+2!+3! + \cdots \\ \nonumber
S_2(n) &= 2!+3! + \cdots  \\ \nonumber
S_3(n) &= 3! + \cdots \\\nonumber
\end{align*}
and 
\begin{align*}
S_0(n) - S_1(n) &= 0!=1\\
S_2(n) - S_1(n) &= 1!=1+1 = 2\\
S_3(n) - S_2(n) &= 2!=1+1+2 = 4\\
S_4(n) - S_3(n) &= 3!=1+ 1+ 2 +6 = 10\\
&\vdots \quad \quad \quad \quad \vdots \quad \quad \quad\vdots \quad \quad \quad
\end{align*}
from Kurepa factorial we know;
\begin{align*}
!n &= 0! +1! + 2! +3! +4! +\cdots+(n-1)!= S_0(n), \quad \mbox{now}
\end{align*}
we know that the geometric series \cite{euler1862opera}
\begin{align*}
\frac{1}{n!}\left(\frac{1}{1-x}\right) &= 1 +x +\frac{x^2}{2!}+ \frac{x^3}{3!} + \cdots +\frac{x^n}{n!} = e^x\\
\end{align*}
when we multiply both sides by the $S_k(n)$ we obtain
\begin{align*}
S_k(n)e^x &=S_k(n) \sum_{n=0}^{\infty} \frac{x^n}{n!}\\
S_k(n)e^x &= S_k(n)\left(1 + x +\frac{x^2}{2!}+ \frac{x^3}{3!}+\cdots +\frac{x^n}{n!}\right)\\
& =  S_k(n) + S_k(n)x +\frac{S_k(n)x^2}{2!}+ \frac{S_k(n)x^3}{3!}+\cdots + \frac{S_k(n)x^n}{n!},\\
\end{align*}
Finally, if $k =0$, $S_0(n)e^x = \mathbf{!n}e^x = \mathbf{Kurepa} \cdotp e^x$
\begin{align*}
\sum_{n=0} S_0(n)\frac{x^n}{n!}&=S_0(n)e^x = S_0(n)\left(1 + x +\frac{x^2}{2!}+ \frac{x^3}{3!}+\cdots  +\frac{x^n}{n!}\right)\\
& =  S_0(n) + S_0(n)x +\frac{S_0(n)x^2}{2!}+ \frac{S_0(n)x^3}{3!}+\cdots + \frac{S_0(n)x^n}{n!}.
\end{align*}
    
\end{proof}

\begin{theorem}\label{recip}
    Consider the sequence $\dfrac{1}{0!},\dfrac{1}{1!} ,\dfrac{1}{2!} , \dfrac{1}{3!}, \dfrac{1}{4!}\cdots$, and the sum of any consecutive $n$ terms $$S_k(n)^{-1}= \frac{1}{k!}+\frac{1}{k+1!}+ \cdots+\frac{1}{(k+n-1)!},$$\\
    setting $k=0$ yields \cite{euler1862opera} $ S_0(n)^{-1}$. 
    \begin{align*}
    S_0(n)^{-1}= \frac{1}{0!}+\frac{1}{1!}+\frac{1}{2!} +\cdots+\frac{1}{(n-1)!}
     \end{align*}
    \\where $\frac{1}{S_k(n)}$ is the reciprocal of $S_k(n)$.
 The product of the inverse exponential series $(\exp{(-x)})$ with the function $S_k(n)^{-1}$ is given by:
\begin{align}
S_k(n)^{-1}e^{-x} &= S_k(n)^{-1}\sum_{n=0}^{\infty}(-1)^n\frac{x^n}{n!}\\
& = S_k(n)^{-1} -S_k(n)^{-1}x + S_k(n)^{-1}\frac{x^2}{2!}- S_k(n)^{-1}\frac{x^3}{3!}+\\ \nonumber
&\cdots + (-1)^nS_k(n)^{-1}\frac{x^n}{n!}\\
& = S_k(1)^{-1}e^{-x} - S_k(2)^{-1}e^{-x} + S_k(3)^{-1}e^{-x}- S_k(4)^{-1}e^{-x}+\\
&\cdots +(-1)^nS_k(n)^{-1}e^{-x}.
\end{align}
If $k=0$ this function becomes 
\begin{align*}
&S_0(n)^{-1}e^{-x} = S_0(n)^{-1}\sum_{n=0}(-1)^n\frac{x^n}{n!}
=\frac{1}{\mathbf{!n}e^x}.
\end{align*}
\end{theorem}

\begin{proof}
Consider the sequence, $\dfrac{1}{0!},\dfrac{1}{1!} ,\dfrac{1}{2!} , \dfrac{1}{3!}, \dfrac{1}{4!}\cdots$, and the sum
\begin{align}\label{Sum}
\frac{1}{s_k(n)} = \frac{1}{k!} +  \frac{1}{(k + 1)!} +  \frac{1}{(k + 2)!} + \cdots +  \frac{1}{(k + n - 1)!}
\end{align}
the product of the exponential series with $S_k(n)^{-1}$ yields;
\begin{align*}
\frac{1}{S_k(n)}e^{-x} &=  \frac{1}{s_k(n)} \left(\frac{1}{0!} -  \frac{x}{1!} + \frac{x^2}{2!} - \frac{x^3}{3!} +\frac{x^4}{4!} + \cdots +(-1)^n\frac{x^n}{n!}\right)\\
&=S_k(n)^{-1} -S_k(n)^{-1}x + S_k(n)^{-1}\frac{x^2}{2!}- S_k(n)^{-1}\frac{x^3}{3!}+ \ldots + (-1)^nS_k(n)^{-1}\frac{x^n}{n!}.
\end{align*}
If we set $k=0$ yields the sum \ref{Sum} becomes,
 $$S_0(n)^{-1}= \frac{1}{0!}+\frac{1}{1!}+\frac{1}{2!} +\cdots+\frac{1}{(n-1)!}$$
and it is easy to see that,
\begin{align*}
S_0(n)^{-1}e^{-x} &= S_0(n)^{-1}\sum_{n=0}(-1)^n\frac{x^n}{n!}\\
&= S_0(n)^{-1} - S_0(n)^{-1}x + S_0(n)^{-1}\frac{x^2}{2!}-\ldots + (-1)^nS_0(n)\frac{x^n}{n!}= \mathbf{!n}^{-1}\exp{(-x)}\\
&= S_0(1)^{-1}e^{-x} + S_0(2)^{-1}e^{-x} + S_0(3)^{-1}e^{-x}+ S_0(4)^{-1}e^{-x}+ \cdots +S_0(n)^{-1}e^{-x}\\
&=\frac{1}{\mathbf{!n }e^x}.
\end{align*}
which completes the proof.    
\end{proof}

\begin{theorem}\label{Sum}\cite{kurepa1971left}
 Consider the sequence $0!, 1!, 2!, 3!, 4!\cdots$, and the sum of any consecutive $n$ terms $$S_k(n)= k!+(k+1)!+ \cdots+(k+n-1)!,$$ naturally the sum
\begin{align*}
        \sum_{r=0}^n S_k(r)x^r= k!+ (k+1)!x +(k+2)!x^2+ (k+3)!x^3+ \cdots+(k+n-1)!x^n
    \end{align*}
satisfies the $S_k(n)$ sum if $x=1$. Also, the series
    \begin{align*}
        \sum_{r=0}^{n} (-1)^{r}S_k(r)x^r= k!-(k+1)!x+ (k+2)!x^2 -(k+3)!x^3+\cdots+(-1)^n(k+n-1)!x^n,
    \end{align*}
Putting $x=1$ naturally yields
 \begin{align*}
        \sum_{r=0}^n (-1)^{r}S_k(r)= k!-(k+1)!+ (k+2)! -(k+3)!+\cdots+(-1)^n(n-1)!
    \end{align*}

If the value of $k=0$ the 
    Kurepa factorial sum, $$\sum_{m=0}^{n-1} m!= 0!+1!+ 2! +3!+ 4!+ \cdots+(n-1)!= \mathbf{!n},$$ naturally satisfies the series(see section \ref{NewConj});
    \begin{align*}
        \sum_{m=0}^n m!x^m= 0!+1!x+ 2!x^2 +3!x^3+ 4!x^4 +\cdots+n!x^n
    \end{align*}
    if $x=1$ we easily obtain just the Kurepa factorials, also, 
    \begin{align}\label{six}
        \sum_{m=0}^{n} (-1)^{m}m!x^m&= 0!-1!x+ 2!x^2 -3!x^3+ 4!x^4 +\cdots+(-1)^n n!x^n.
    \end{align}
\end{theorem}

\begin{proof}
Let $0!, 1!, 2!, 3!, 4!\cdots$, be kurepa sequence and the sum of any consecutive $n$ terms given by $$S_k(n)= k!+(k+1)!+ \cdots+(k+n-1)!,$$ 
one can write
  \begin{align*}
        \sum_{n=0} S_k(n)x^n= k!+ (k+1)!x +(k+2)!x^2+ (k+3)!x^3 +\cdots+(k+n-1)!x^n,
    \end{align*}  
it is trivial to obtain $ S_k(n)$ when setting $x=1.$ 
 \begin{align*}\label{Kup}
        \sum (-1)^{n}S_k(n)x^n= k!-(k+1)!x+ (k+2)!x^2 -(k+3)!x^3+\cdots+(-1)^n(k+n-1)!x^n,
    \end{align*}
setting $x=1$ yields the $ \sum (-1)^{n}S_k(n)$  which we shall discuss in subsequent theorems. Also, when $k=0$ we observe that
    \begin{align*}
        \sum_{m=0}^{n} m!x^m= 0!+1!x+ 2!x^2 +3!x^3+ 4!x^4 +\cdots+n!x^n,
    \end{align*}
    where $\sum_{m=0}^{n-1} m!= 0! +1! + 2! +3!+ 4! +\cdots+(n-1)!= \mathbf{!n}$ is the Kurepa sum. Also, in equation \ref{six}, if $k=1$ we easily notice that
    \begin{align*}
        \sum_{m=0}^{n} (-1)^{m}m!x^m= 0!-1!x+ 2!x^2 -3!x^3+ 4!x^4 +\cdots+(-1)^n n!x^n.
    \end{align*}
    proof completed.
    

\end{proof}



\subsection{Kurepa Sequence}
The Kurepa factorial has become a very interesting concept that has drawn much attention over the past 5 decades, authors like Don Zagier and Sun, Anne Gertsch and many more \cite{sun2011curious, gertsch1999congruences, gertsch1996some} has shown the connections btween the Kurepa factorial for primes to the Bell number, Derangement number and many more. In this subsection, I seek to investigate more the Kurepa factorials and to answer to some extent the questions posed in the introduction \ref{Solve}.

\begin{definition}\label{Seq}
    For all $n\in \mathbb{N}$, let  
    \begin{align*}
        \lbrace K_n\rbrace_{n\geq 1}&=\sum_{i=1}^{n}K_i=K_1 + K_2 + K_3 +K_4 +\cdots + K_n
    \end{align*}
 be the Kurepa sequence, 
where  $$K_n=\mathbf{!n}=\sum_{m=0}^{n-1}m!= 0! + 1!+ 2! + 3! + 4!+ 5!  + \cdots +(n-1)!=S_0(n).$$ with $m<n$.
\end{definition}
\begin{theorem}\label{dobi}
For the series $\lbrace K_n\rbrace_{n\geq 1}\cdotp e^{x}$, the product,
 \begin{align*}
     \lbrace K_n\rbrace_{n\geq 1}\cdotp e^{x}&= (K_1 + K_2 + K_3 +K_4 +\cdots + K_n)\cdotp e^{x}
 \end{align*}
if we set $x=1$, then
 \begin{align*}
     \lbrace K_n\rbrace_{n\geq 1}\cdotp e&= (K_1e + K_2e + K_3e +K_4e +\cdots + K_ne)\\
     & =!1e + !2e + !3e + \cdots + !ne
 \end{align*}
 where 
$\mathbf{!n} \cdotp e = \left(0! +1! + 2! + 3! + \cdots + (n-1)!\right) e $.

\end{theorem}

\begin{proof}
From definition \ref{Seq} and theorem \ref{Add} the proof of this is straightforward.
    
\end{proof}

\begin{theorem}\label{Dk}
The series $ \lbrace K_n\rbrace_{n\geq 1}\cdotp e$ is the sum of the Dobinski numbers $(\mathbf{Dob_n})$, that is,
$$\lbrace K_n\rbrace_{n\geq 1}\cdotp e= \sum_{r=0}^{n} \Phi_r \mathbf{Dob_r}$$ where $\Phi_r$ is coefficients(constant). We remark that $\mathbf{Dob_0}=\mathbf{Dob_1}=e$, so starting $r$ at $0$ or $1$ does not change the equation. The table \ref{table 2} below gives some few partitions:
\begin{table}[!ht]
    \centering
    \begin{tabular}{|c|a|c|c|c|c|c|c|c|c|}\hline 
        $\mathbf{!n\cdotp \exp(1)}/ \mathbf{Dob_n}$& $\mathbf{Dob_0}$ & $\mathbf{Dob_1}$ & $\mathbf{Dob_2}$ & $\mathbf{Dob_3}$ &$\mathbf{Dob_4}$  &$\mathbf{Dob_5}$ & $\mathbf{Dob_6}$ \\\hline  \rowcolor{LightCyan}
        $\mathbf{!1e}=1e$ &  & $\mathbf{Dob_1}$ &  &  &  &  &   \\\hline
         $\mathbf{!2e}=2e$ &  &  & $\mathbf{Dob_2}$ &  &  &  &   \\\hline
         $\mathbf{!3e}=4e$ &  &  & $2\mathbf{Dob_2}$ &  &  &  &  \\\hline
          $\mathbf{!4e}=10e$ &  &  &  &  $2\mathbf{Dob_3}$&  &  &  \\\hline
        $\mathbf{!5e}=34e$ &  &  & $2\mathbf{Dob_2}$ &  & $2\mathbf{Dob_4}$ &  &  \\\hline
        $\mathbf{!6e}=154e$ &  &  &$2\mathbf{Dob_2}$  & & $10\mathbf{Dob_4}$ &  &  \\\hline
         $\mathbf{!7e}=874e$&  &  & $\mathbf{Dob_2}$ &  &$4\mathbf{Dob_4}$  &  & $4\mathbf{Dob_6}$  \\\hline
        $\mathbf{!8e}=5914e$ &  &  &  & &$40\mathbf{Dob_4}$ & $\mathbf{Dob_5}$ &   \\\hline
    \end{tabular}
    \caption{Kurepa and Dobinski number}
    \label{table 2}
\end{table}
\end{theorem}

\begin{proof}
From theorem \ref{dobi} and definition \ref{Seq},
\begin{align*}
 \lbrace K_n\rbrace_{n\geq 1}\cdotp e^{x}&= (K_1 + K_2 + K_3 +K_4 +\cdots + K_n)\cdotp e^{x}\\
 &=K_1\cdotp e^{x} + K_2 \cdotp e^{x}+ K_3 \cdotp e^{x}+K_4 \cdotp e^{x}+\cdots + K_n \cdotp e^{x}\\ 
 \lbrace K_n\rbrace_{n\geq 1}\cdotp e &=K_1\cdotp e + K_2 \cdotp e + K_3 \cdotp e +K_4 \cdotp e +\cdots + K_n \cdotp e\\ 
&= !1\cdotp e +!2\cdotp e + !3\cdotp e + !4\cdotp e + \cdots + \mathbf{!n}\cdotp e\\ \mbox{where} \quad
& K_1e=!1e = 1\cdotp e = \mathbf{Dob_1}= \mathbf{Dob_0}=e\\
& K_2e=!2e = 2\cdotp e = \mathbf{Dob_2}\\
&K_3e=!3e = 4\cdotp e = 2(2e) =  2\mathbf{Dob_2}\\
&K_4e=!4e = 10 \cdotp e = 2(5e) =  2\mathbf{Dob_3}\\
& K_5e=!5 e = 34 \cdotp e = 2(15e+2e) =  2\mathbf{Dob_4}+ 2\mathbf{Dob_2}\\
&K_6e=!6 e = 154e = 10(15e) + 4e = \mathbf{10 Dob_4} + 2 \mathbf{Dob_2}\\
&K_7e=!7e = 874e = 4(203e) + 4(15e) + 2e = 4 \mathbf{Dob_6} + 4 \mathbf{Dob_4} + \mathbf{Dob_2}\\
& K_8e=!8e = 5914e = 40(15e) + 52e=40 \mathbf{Dob_4} + \mathbf{Dob_5}\\
&\vdots \quad \quad \quad \quad \vdots \quad \quad \quad\vdots \quad \quad \quad \quad \quad \quad \quad\vdots \quad \quad
\end{align*}
we take $n=8$ this leads to
\begin{align*}
 \lbrace K_8\rbrace_{n\geq 1}\cdotp \exp(1) &= 0!\cdotp e +1!\cdotp e + 2!\cdotp e + 3!\cdotp e + \ldots + (n-1)!\cdotp e\\
 \lbrace K_8\rbrace_{n\geq 1}\cdotp e &= 1 e + 2e + 4e + 10e+34e+ 154e+ 874e+ 5914e \cdots\\
 &= 1 e + 2e + 2(2e) + 2(5e) +2(15e+2e)+ 10(15e) + 4e +  4(203e)\\
 &+ 4(15e) + 2e +  40(15e) + 52e\\
&=e +8(2e) + 2(5e)+56(15e)+ 52e + 4(203e) \\ \mbox{which yields} \quad\\
 \lbrace K_8\rbrace_{n\geq 1}\cdotp e&= \mathbf{Dob_1} + 8\mathbf{Dob_2} + 2\mathbf{Dob_3} +56\mathbf{Dob_4} + \mathbf{Dob_5}+ 4\mathbf{Dob_6}
\end{align*}
we notice this sequence depends on the value of $n$ to determine the coefficients $\Phi$, thus this completes the proof.  
\end{proof}

\begin{theorem}\label{Belo}
The Kurepa sequence  $\lbrace K_n\rbrace_{n\geq 1}$ is the sum of the Bell numbers $ \mathbf{Bell_n}$. 
\end{theorem}

\begin{proof}
From Theorem \ref{Dk} and using $n=8$ we have
    \begin{align}
        \lbrace K_8\rbrace_{n\geq 1}\cdotp e &= \mathbf{Dob_1} + 8\mathbf{Dob_2} + 2\mathbf{Dob_3} +56\mathbf{Dob_4} + \mathbf{Dob_5}+ 4\mathbf{Dob_6}\\ \nonumber
       \lbrace K_8\rbrace_{n\geq 1} &= \dfrac{ \mathbf{Dob_1} + 8\mathbf{Dob_2} + 2\mathbf{Dob_3} +56\mathbf{Dob_4} + \mathbf{Dob_5}+ 4\mathbf{Dob_6}}{\exp(1)}\\ \nonumber
         \lbrace K_8\rbrace_{n\geq 1} &= \dfrac{\mathbf{Dob_1} }{e} + 8\dfrac{\mathbf{Dob_2}}{e} +2\dfrac{\mathbf{Dob_3}}{e}+ 56\dfrac{\mathbf{Dob_4}}{e}+ \dfrac{\mathbf{Dob_5}}{e}+ 4\dfrac{\mathbf{Dob_6}}{e}\\
    &= \mathbf{Bell_1}+ 5 \mathbf{Bell_2}  + 2 \mathbf{Bell_3} + 56\mathbf{Bell_4} + \mathbf{Bell_5} + 4\mathbf{Bell_6} \quad
\mbox{thus}\\ \nonumber
 \lbrace K_8\rbrace_{n\geq 1} &= \mathbf{Bell_1}+ 8 \mathbf{Bell_2}  + 2 \mathbf{Bell_3} + 56\mathbf{Bell_4}+ \mathbf{Bell_5} + 4\mathbf{Bell_6} \\
 &\mbox{there are coefficients constant that depends on the vlaue of $n$}.
    \end{align}
    Also, it is easy to see that the Bell numbers can be expressed in Stirling numbers of the second kind, thus
    $$\lbrace K_8\rbrace_{n\geq 1}= \sum S(1,k) + 8  \sum S(2,k)  + 2  \sum S(3,k)+ 56  \sum S(4,k)+ \sum S(5,k) + 4  \sum S(6,k) $$
\end{proof}

\begin{theorem}\label{ok}
    The product of the Kurepa sequence $ \lbrace K_n\rbrace_{n\geq 1}$ and the ordinary factorial numbers $n!$ is the sum of the product of the derangement numbers with the Dobinski numbers, that is, 
        $$\lbrace K_n\rbrace_{n\geq 1} \cdotp n!= n!\sum_{r=0}^{n} \dfrac{\Phi_r \mathbf{Dob_r}}{e} = \dfrac{n!}{e}  \sum_{r=0}^{n} \Phi_r \mathbf{Dob_r}=\sum_{r=0}^{n} \Phi_r (\mathbf{Der_r \cdotp Dob_r})$$
        where $\Phi_r$ is coefficient(constant) of the $\mathbf{Der_n \cdotp Dob_n}$ .
\end{theorem}

\begin{proof}
Let
\begin{align*}
n!e^{-x}& = n!\left( 1-x + \frac{x^{2}}{2!} - \frac{x^3}{3!} + \frac{x^4}{4!} + \cdots + (-1)^i\frac{x^i}{i!}\right)\mbox{if $x=1$ then}\\
n!e^{-1}& = n!\left( 1- 1 + \frac{1^{2}}{2!} - \frac{1^3}{3!} + \frac{1^4}{4!} + \cdots + \frac{(-1)^i}{i!}\right)\\
\frac{n!}{e} &= n!\left( 1-1 + \frac{1^2}{2!} - \frac{1^3}{3!} + \frac{1^4}{4!}+ \cdots +\frac{(-1)^i}{i!} \right)\\
&\mbox{it is well known that, the derangement \cite{clarke1993derangements}}\\
\mathbf{Der_n}&= \frac{n!}{e} = n! \left( 1-1 + \frac{1^2}{2!} - \frac{1^3}{3!} + \frac{1^4}{4!}+ \cdots +\frac{(-1)^i}{i!} \right)\\
\mathbf{Der_n}&= \frac{n!}{e} = n!\sum_{i=0}^n\frac{(-1)^i}{i!} = n!e^{-1}
\end{align*}
from Theorem \ref{Belo} and for $n=8$
  \begin{align*}
      \textbf{Kurepa sequence} \cdotp n!&=  \lbrace K_8\rbrace_{n\geq 1} \cdotp n! \\
      &=n! \left(  \dfrac{\mathbf{Dob_1} }{e} + 8\dfrac{\mathbf{Dob_2}}{e} +2\dfrac{\mathbf{Dob_3}}{e}+ 56\dfrac{\mathbf{Dob_4}}{e}+ \dfrac{\mathbf{Dob_5}}{e}+ 4\dfrac{\mathbf{Dob_6}}{e}\right) \\
   &=  n!\dfrac{\mathbf{Dob_1} }{e} + 8\cdotp n!\dfrac{\mathbf{Dob_2}}{e} +2\cdotp n!\dfrac{\mathbf{Dob_3}}{e}+ 56\cdotp n!\dfrac{\mathbf{Dob_4}}{e}\\
  & + n!\dfrac{\mathbf{Dob_5}}{e}+ 4\cdotp n!\dfrac{\mathbf{Dob_6}}{e}\\
   &= n!e^{-1}\mathbf{Dob_1} + 8(n!e^{-1})\mathbf{Dob_2} + 2(n!e^{-1})\mathbf{Dob_3} + 56(n!e^{-1})\mathbf{Dob_4} \\
  & \quad + (n!e^{-1})\mathbf{Dob_5}+ 4(n!e^{-1})\mathbf{Dob_6}\\
   &= \mathbf{Der_1} \mathbf{Dob_1} + 8\mathbf{Der_2} \mathbf{Dob_2} + 2 \mathbf{Der_3} \mathbf{Dob_3} +56 \mathbf{Der_4} \mathbf{Dob_4} \\ 
  &\quad+ \mathbf{Der_5} \mathbf{Dob_5}+ 4 \mathbf{Der_6} \mathbf{Dob_6}
    \end{align*}
the proof is immediate.
\end{proof}

\begin{theorem}\label{nine}
    The product of the ordinary factorial numbers $k!$ and the sum of Bell numbers $\mathbf{Bell_n}$ is the sum of the product of the derangement numbers with the Dobinski numbers $\mathbf{Der_n \cdotp Dob_n}$, that is,
    $$ k! \sum_{k=1}^{n}  \mathbf{Bell_k}=\sum_{r=1}^{n} \Phi_r(\mathbf{Der_r \cdotp Dob_r}).$$
\end{theorem}

\begin{proof}
From Theorem \ref{Sum} and Theorem \ref{ok}, we can see that,
\begin{align*}
k! \mathbf{Bell_n}&= k!(e^{-1})\mathbf{Dob_n}\\
k! \mathbf{Bell_n}&= k! e^{-1}\mathbf{Dob_n} = \mathbf{Der_n}\cdotp \mathbf{Dob_n}\quad (\mbox{for the sum of $\mathbf{Bell_n}$})\\
   k! \sum_{r=1}^{n} \mathbf{Bell_r}&=k!\left(\mathbf{Bell_1}+ 8 \mathbf{Bell_2}  + 2 \mathbf{Bell_3} + 56\mathbf{Bell_4} + \mathbf{Bell_5} + 4\mathbf{Bell_6} \right)\\
&= k!e^{-1}\mathbf{Dob_1} + 8(k!e^{-1})\mathbf{Dob_2} + 2(k!e^{-1})\mathbf{Dob_3} + 56(k!e^{-1})\mathbf{Dob_4}\\ \nonumber
 & \quad + (n!e^{-1})\mathbf{Dob_5}+ 4(n!e^{-1})\mathbf{Dob_6}\\
&= \mathbf{Der_1} \mathbf{Dob_1} + 8\mathbf{Der_2} \mathbf{Dob_2} + 2 \mathbf{Der_3} \mathbf{Dob_3} +56 \mathbf{Der_4} \mathbf{Dob_4}\\
 &\quad+ \mathbf{Der_5} \mathbf{Dob_5}+ 4 \mathbf{Der_6} \mathbf{Dob_6}
    \end{align*}
this finishes the proof.    
\end{proof}
The following consequence is immediate as a corollary;
\begin{corollary}
    For all nonnegative integers $n\quad\mbox{and}\quad k$,
    $$\lbrace K_n\rbrace_{n\geq 1} \cdotp n!=k! \sum_{r=1}^{n}\mathbf{Bell_r}.$$
\end{corollary}
\begin{proof}
    From Theorem \ref{ok} and Theorem \ref{nine} the proof of this is trivial.
\end{proof}
\begin{theorem}\label{Kprove}
    The Kurepa sequence $$ \lbrace K_n\rbrace_{n\geq 1}=\sum_{r=1}^{n} \Phi_r \mathbf{Bell_r}=\sum_{r=1}^{n} \Phi_r \sum_{k\geq 1}^{r} S(r,k),$$ where $S(r,k)$ is  the Stirling numbers of the second kind and the table \ref{table 3} below shows some few partition sequence.
We remark that $\mathbf{Bell_0}=\mathbf{Bell_1}=1$, so starting $r=1$ does not change the equation.
\begin{table}[!ht]
    \centering
    \begin{tabular}{|c|a|c|c|c|c|c|c|c|c|}\hline 
        $\mathbf{!n/ \mathbf{Bell_n}}$ & $\mathbf{Bell_0}$ & $\mathbf{Bell_1}$ & $\mathbf{Bell_2}$ & $\mathbf{Bell_3}$ &$\mathbf{Bell_4}$  &$\mathbf{Bell_5}$ & $\mathbf{Bell_6}$ \\\hline  \rowcolor{LightCyan}
        $\mathbf{!1}=1$ &  & $\mathbf{Bell_1}$ &  &  &  &  &   \\\hline
         $\mathbf{!2}=2$ &  &  & $\mathbf{Bell_2}$ &  &  &  &   \\\hline
         $\mathbf{!3}=4$ &  &  & $2\mathbf{Bell_2}$ &  &  &  &  \\\hline
          $\mathbf{!4}=10$ &  &  &  &  $2\mathbf{Bell_3}$&  &  &  \\\hline
        $\mathbf{!5}=34$ &  &  & $2\mathbf{Bell_2}$ &  & $2\mathbf{Bell_4}$ &  &  \\\hline
        $\mathbf{!6}=154$ &  &  &$2\mathbf{Bell_2}$  & & $10\mathbf{Bell_4}$ &  &  \\\hline
         $\mathbf{!7}=874$&  &  & $\mathbf{Bell_2}$ &  &$4\mathbf{Bell_4}$  &  & $4\mathbf{Bell_6}$  \\\hline
        $\mathbf{!8}=5914$ &  &  &  & &$40\mathbf{Bell_4}$ & $\mathbf{Bell_5}$ &   \\\hline
    \end{tabular}
    \caption{Kurepa and Bell number}
    \label{table 3}
\end{table}

\end{theorem}

\subsection{Shifted alternating Kurepa sequence}
Miodrag zivkovič \cite{zivkovic1999number} the number of primes of the type $A_n$ is finite, since for $n\geq p_1$,  $A_n$ is divisible by $p_1$. The heuristic argument posits the existence of a prime $p$ such that $p$ divides $!n$ for any large $n$, nevertheless, computational verification indicates that this prime must exceed $2^{23}$. Due to the connection this has with the Kurepa factorial, authors such as, Kevin Buzzard, Alexandar Petojevic, Z. Mijajlovic and many more \cite{mijajlovic1990some, mijajlovic2021fifty, petojevic2023new, andrejic2016searching, sun2011curious} have done extensive works in this field.
In Guy`s book of unsolved problems \cite{guy2004unsolved}, the alternating sums of factorials is given as follows
$$A_{n+1}=\sum_{m=1}^{n}(-1)^{n-m} m!,$$
there are questions if $0!$ is included. The numbers are now even, and only $2! - 1! + 0! = 2$ is prime; this makes it more interesting in the subsequent results that we have as this reveals much information about the shifted alternating Kurepa introduced in the subsequent section.
In \cite{zivkovic1999number}, Miodrag used Wagstaff definition of the Kurepa factorial, that is, $!n -1$ which yields the values in table \ref{mykure}.
Note that Wagstaff \cite{zivkovic1999number, mudge1996introducing, ashbacher1997some, cira2016various} verified the Kurepa conjecture for $n<50000.$

\begin{table}[!ht]\label{M Kure}
\centering
\begin{tabular}{|c|c|c|c|c|}\hline
$n$ & $A_n^s = \sum_{m=0}^{n-1} (-1)^m m!$ & $K_n =!n= \sum_{m=0}^{n-1} m!$ & $A_{n+1}$& $WK_n=!n-1$\\
\hline
0 & 0 & 0 & 0& 0\\
1 & 1 & 1 & 1 & 0\\ \rowcolor{LightCyan}
2 & 0 & 2 & 1 & 1\\ 
3 & 2 & 4 & 5 & 3\\
4 & -4 & 10 & 19& 9\\
5 & 20 & 34 & 101& 33\\
6 & -100 & 154& 619& 153 \\
7 & 620 & 874& 4421&873 \\
8 & -4420 & 5914 &35899& 5913 \\
9 & 35900 & 46234& 326981& 46233 \\
10 & -326980 & 409114& 3301819& 409113 \\
\hline
\end{tabular}
  \caption{Kurepa and  alternating sum of factorials}
    \label{mykure}
\end{table}

\begin{definition}\label{Alt}
    For all $n\in \mathbb{N}$, let  
    \begin{align*}
      \lbrace A_n^s\rbrace_{n\geq 1}=\sum_{i=1}^{n}A_i^s=A_1^s + A_2^s + A_3^s +A_4^s +\cdots + A_n^s
    \end{align*}
 be the shifted alternating Kurepa sequence (see table \ref{mykure}), 
where  $$A_n^s=\mathbf{(-1)^n\cdotp !n}=\sum_{m=0}^{n-1}(-1)^m m!= 0! - 1!+ 2! -3! + 4!- 5!  + \cdots +(-1)^{n-1}(n-1)!$$ with $m<n$ \cite{guy2004unsolved, weisstein2005alternating, zivkovic1999number, mijajlovic1990some}.
\end{definition}

\begin{theorem}\label{Aliter}
The shifted alternating Kurepa sequence, $\lbrace A_n^s\rbrace_{n\geq 1}$, is the sum of complementary Bell numbers,
$$  \lbrace A_n^s\rbrace_{n\geq 1}=\sum_{r=0}^{n}\Phi_r(\mathbf{invBell_r})= \sum_{r=1}^{n}\Phi_r \sum_{k\geq1}^{r} (-1)^kS(n,k). $$
\end{theorem}
\begin{proof}
From definition \ref{Alt}, and table \ref{M Kure};
\begin{align*}
A_n^s=&(-1)^n\mathbf{!n}=\sum_{m=0}^{n-1}(-1)^{m}m!= !0 - 1!+ !2 - 3!  + \cdots + (-1)^{n-1}(n-1)!\\ \mbox{where}\\
 A_1^s=&(-1)^1\mathbf{!1}=\sum_{1}(-1)^{1-1}0!= 1\\
A_2^s=&(-1)^2\mathbf{!2}=\sum_{1} (-1)^{1-1}0!+ \sum_{2}(-1)^{2-1}1!= 1-1=0\\
A_3=&(-1)^3\mathbf{!3}= \sum_{3}(-1)^{3-1}2!+ \sum_{2}(-1)^{3-2}1!+\sum_{3} (-1)^{3-3}0!= 2-1+1=2\\
A_4^s=&(-1)^n\mathbf{!4}=-6+2-1+1=-4\\
A_5^s=&(-1)^n\mathbf{!5}=24-6+2-1+1=20\\
A_6^s=&(-1)^n\mathbf{!6}=-120+24-6+2-1+1=-100\\
A_7^s=&(-1)^n\mathbf{!7}=720-120+24-6+2-1+1=620\\
A_8^s=&(-1)^n\mathbf{!8}=-5040+720-120+24-6+2-1+1=-4420\\
A_9^s=&(-1)^n\mathbf{!9}=40320-5040+720-120+24-6+2-1+1=35900\\
&\quad \quad \vdots\quad \quad \quad \quad \quad \vdots \quad \quad \quad\vdots \quad \quad \quad \quad \quad \quad \quad \quad  \quad\vdots \quad \quad\\
\mbox{Now the sequence}\\
\lbrace A_n^s\rbrace_{n\geq 1}&=\sum_{i=1}^{n}A_i^s=A_1^s + A_2^s + A_3^s +A_4^s +\cdots + A_n^s\\
\lbrace A_n^s\rbrace_{n\geq 1}\cdotp e^{-x}&= (A_1^s + A_2^s + A_3^s +A_4^s +\cdots + A_n^s)\cdotp e^{-x}\\
&= A_1^s\cdotp e^{-x} + A_2^s \cdotp e^{-x}+ A_3^s \cdotp e^{-x}+A_4^s \cdotp e^{-x}+\cdots + A_n^s \cdotp e^{-x}\\ 
\mbox{if $x=1$ we obtain}\\
\lbrace A_n^s\rbrace_{n\geq 1}\cdotp e^{-1}&= (-1)^0 !1\cdotp e^{-1} + (-1)^1 !2\cdotp e^{-1} + (-1)^2 !3\cdotp e^{-1}\\ 
&+ (-1)^3 !4\cdotp e^{-1} +\cdots + (-1)^n A_n^s e^{-1}\\
&= 1\cdotp e^{-1} -  0\cdotp e^{-1} + 2\cdotp e^{-1} - 4\cdotp e^{-1} + 20\cdotp e^{-1} -100\cdotp e^{-1}+ 620\cdotp e^{-1}\\
&\quad -\cdots+(-1)^{n-1}(n-1)!e^{-1}\\
\end{align*}
If $n=5$ by simple computations we arrive at:
\begin{align*}
\lbrace A_5^s\rbrace_{n\geq 1}\cdotp e^{-1}&=\mathbf{invDob_0} + \mathbf{invDob_2} + 2\cdotp \mathbf{invDob_3} +2\cdotp \mathbf{invDob_5} + 20\cdotp \mathbf{invDob_4} \\
&\quad + 50\cdotp \mathbf{invDob_{5}}.\\
\mbox{We observe that}\\
 &\lbrace A_5^s\rbrace_{n\geq 1}\cdotp e^{-1}=\sum_{r=0}^{5}\Phi_5 \mathbf{invDob_5}\\
 &\mbox{now the shifted alternating Kurepa sequence becomes}\\
\lbrace A_5^s\rbrace_{n\geq 1}&=\mathbf{invDob_0}e + \mathbf{invDob_2}e + 2\cdotp \mathbf{invDob_3}e +2\cdotp \mathbf{invDob_5}e + 20\cdotp \mathbf{invDob_4}e\\
&\quad   + 50\cdotp \mathbf{invDob_{5}}e\\
\lbrace A_5^s\rbrace_{n\geq 1}&=\mathbf{invBell_0} + \mathbf{invBell_2} + 2\cdotp \mathbf{invBell_3} +2\cdotp \mathbf{invBell_5} + 20\cdotp \mathbf{invBell_4} \\
&\quad   + 50\cdotp \mathbf{invBell_{5}}
\end{align*}
the proof immediately follows.
\end{proof}

\begin{theorem}\label{ok}
    The product of shifted alternating Kurepa sequence with ordinary factorial numbers $n!$ is the sum of the product of derangement numbers and the complementary Bell numbers, that is, $\mathbf{Der_n}\cdotp(a_n)$ where $a_n= \mathbf{invBell_n}$ for all nonnegative integers $n$, that is, 
    $$ \mathbf{Der_n}\lbrace A_5^s\rbrace_{n\geq 1}=\sum_{k=0}^{n} \Phi_k \cdotp\mathbf{Der}_{k}\cdotp(a_k).$$
\end{theorem}
\begin{proof}
it is well known from definition \ref{id} that,

 $$\mathbf{invDob}_n=\sum_{k=0}^{\infty}(-1)^k\dfrac{k^n}{k!}. $$
Now multiplying through by ordinary $n!$ yields 
\begin{align}
    n!(\mathbf{invDob}_n)&= n!\left(\sum_{k=0}^{\infty}(-1)^k\dfrac{k^n}{k!}\right)\\
     &= n!\sum_{k=0}^{n}\frac{(-1)^k}{k!}k^n=\dfrac{n!}{e}\sum_{k=0}^{n}(-1)^k S(n,k)\\
     &= \mathbf{Der}_n\left(\mathbf{invBell_n}\right)=\mathbf{Der}_n\cdotp(a_n)
\end{align}
with $a_n=\mathbf{invBell_n}$ the few Derangement polynomials with respect to $k$ are given below;
\begin{align}
    \dfrac{-1}{e}n!&= n!\sum_{n=1}(-1)^k\dfrac{k}{k!}=\mathbf{Der}_1 \cdotp(a_1)\\ \nonumber
   \dfrac{0}{e}n!&=n!\sum_{n=2}(-1)^k\dfrac{k^2}{k!}=\mathbf{Der}_2 \cdotp(a_2)\\ \nonumber
    \dfrac{1}{e}n!&=n!\sum_{n=3}(-1)^k\dfrac{k^3}{k!}=\mathbf{Der}_3 \cdotp(a_3)\\ \nonumber
    \dfrac{1}{e}n!&= n!\sum_{n=4}(-1)^k\dfrac{k^4}{k!}=\mathbf{Der}_4 \cdotp(a_4)\\ \nonumber
    \dfrac{-2}{e}n!&=n!\sum_{n=5}(-1)^k\dfrac{k^5}{k!}=\mathbf{Der}_5 \cdotp(a_5)\\ \nonumber
   \dfrac{-9}{e}n!&= n!\sum_{n=6}(-1)^k\dfrac{k^6}{k!}=\mathbf{Der}_6 \cdotp(a_6)\\ \nonumber
    \dfrac{-9}{e}n!&=n!\sum_{n=7}(-1)^k\dfrac{k^7}{k!}=\mathbf{Der}_7 \cdotp(a_7) \\ \nonumber
    &\vdots \quad \quad \quad \quad \vdots \quad \quad \quad\vdots \quad \quad \quad \quad \quad \quad \quad\\
    \dfrac{\mathbf{invBell_{n}}}{e}n!&=n!\sum_{n=1}^{\infty}(-1)^k\dfrac{k^n}{k!}=\mathbf{Der}_{n}\cdotp(a_n)
\end{align}
from Theorem \ref{Aliter} we have $n=5$
\begin{align*}
   n! \dfrac{\lbrace A_5^s\rbrace_{n\geq 1}}{e} &= n!\left(\mathbf{invDob_0} + \mathbf{invDob_2} + 2\cdotp \mathbf{invDob_3} +2\cdotp \mathbf{invDob_5} \right. \\
&\quad  \left.  + 20\cdotp \mathbf{invDob_4} + 50\cdotp \mathbf{invDob_{5}}\right) \\
  \mathbf{Der_n}\lbrace A_5^s\rbrace_{n\geq 1}&= \mathbf{Der_0}\cdotp(a_0) + \mathbf{Der_2}\cdotp(a_2) + 2\cdotp \mathbf{Der_4}\cdotp(a_4) +2\cdotp \mathbf{Der_5}\cdotp(a_5) \\
  & \quad + 20\cdotp \mathbf{Der_4}\cdotp(a_4) + 50\cdotp \mathbf{Der_{5}}\cdotp(a_5)\\
\end{align*}
hence proof easily follows immediately and thus completed.
\end{proof}
\begin{theorem}\label{div}
    The product of the ordinary factorial numbers $n!$ and the sum of complementary Bell numbers $\mathbf{Bell_n}$ is the sum of product of the derangement numbers, Dobinski numbers, and complementary Bell numbers; $\mathbf{Der_n \cdotp Dob_n \cdotp \mathbf{invBell_n}}$, that is,
    $$k! \lbrace A_n^s\rbrace_{n\geq 1}=\sum_{r=0}^{n} \Phi \mathbf{Der_r \cdotp Dob_r}\cdotp \mathbf{invBell_n}$$
\end{theorem}
\begin{proof}

\begin{align*}
e^{(1-e^x)} & = e\cdotp e^{-e^x} = e( \mathbf{invDob_n})\\
&= e\cdotp n!\sum^{\infty}_{k=0}\frac{(-1)^k}{k!}k^n
 \exp(x)=n!\cdotp \mathbf{C_k} \cdotp \exp{(x)}\\
 & = e\mathbf{Der_k}\cdotp(a_n) = n!\mathbf{C_k} \cdotp \exp{(x)}= n!(\mathbf{inv Bell_n})
 \end{align*}
from Theorem \ref{Aliter} and Theorem \ref{div} for $n=5$
\begin{align*}
   k! \lbrace A_5^s\rbrace_{n\geq 1}&= k! \left(\mathbf{invBell_0} + \mathbf{invBell_2} + 2\cdotp \mathbf{invBell_3} +2\cdotp \mathbf{invBell_5} + 20\cdotp \mathbf{invBell_4}\right.\\ &\quad\left.+ 50\cdotp \mathbf{invBell_{5}}\right)\\
   &=e\cdotp \mathbf{Der_0}\cdotp(a_0) + e\cdotp\mathbf{Der_2}\cdotp(a_2) + 2e\cdotp \mathbf{Der_4}\cdotp(a_4) +2e\cdotp \mathbf{Der_5}(a_5) + 20e\cdotp \mathbf{Der_4}\cdotp(a_4) \\
   & \quad + 50e\cdotp \mathbf{Der_{5}}\cdotp(a_5)\\
    &=e\cdotp \mathbf{Der_0}\cdotp(a_0) + e\cdotp\mathbf{Der_2}\cdotp(a_2) + 2e\cdotp \mathbf{Der_4}(a_4) +2e\cdotp \mathbf{Der_5}(a_5) + (15e+ 5e)\cdotp \mathbf{Der_4}(a_4) \\\nonumber
    &+ (3(15e)+5e)\cdotp \mathbf{Der_{5}}\cdotp(a_5)\\\nonumber
  &=\mathbf{Dob_0}\cdotp \mathbf{Der_0}\cdotp(a_0) + \mathbf{Dob_1}\cdotp\mathbf{Der_2}(a_2) + \mathbf{Dob_2}\cdotp \mathbf{Der_4}\cdotp(a_4) + \mathbf{Dob_2}\cdotp \mathbf{Der_5}\cdotp(a_5) \\\nonumber
 & + (\mathbf{Dob_4}+ \mathbf{Dob_3})\cdotp \mathbf{Der_4}\cdotp(a_4)+ (3\mathbf{Dob_4}+ \mathbf{Dob_3})\cdotp \mathbf{Der_{5}}\cdotp(a_5)
\end{align*}
\end{proof}

\begin{lemma}
    The polynomial function $\mathbf{!n}e^x$ has a reciprocal function of $(\mathbf{!n} e^x)^{-1}$.
\end{lemma}
\begin{proof} From Theorem \ref{recip}, one obtains;
   \begin{align*}
!ne^x = !ne^x &= e^x0!\quad \quad \mbox{ if } x=1 \quad\quad !n e^1 = e\cdotp0! =\mathbf{Dob_1}\\
&=e^x1!\quad \quad \mbox{ if } x=1, \quad\quad !n e^1\cdotp1! = e\cdotp1! =2e\\
&=e^x2!\quad \quad \mbox{ if } x=1, \quad\quad !n e^1\cdotp 2! = e\cdotp 2! =4e\\
&=e^x3! \quad \quad \mbox{ if } x=1, \quad\quad !n e^1\cdotp 3! = e\cdotp 3! =10e\\
&=e^x4!\quad \quad \mbox{ if } x=1, \quad\quad !n e^1\cdotp 4! = e\cdotp 4! =34e,\\
&\quad \quad \vdots\quad \quad \quad \quad \quad \vdots \quad \quad \quad\vdots \quad \quad \quad \quad \quad \quad \quad  \quad\vdots \quad \quad
\end{align*}
similarly,
\begin{align*}
(!n e^x)^{-1} &= \frac{1}{!ne^x} \\
& = \frac{1}{!ne^1}= \frac{1}{e\cdotp 0!} = \frac{1}{\mathbf{Dob_1}} = \frac{1}{e} \\
&=\frac{1}{e\cdotp 1!} =\frac{1}{2e} =\frac{1}{2e}=\frac{1}{\mathbf{Dob_2}}\\
&= \frac{1}{e\cdotp 2!} =\frac{1}{4e}=\frac{1}{2(2e)}=\frac{1}{2\mathbf{Dob_2}}\\
& = \frac{1}{e\cdotp 3!} =\frac{1}{10e}=\frac{1}{2(5e)}=\frac{1}{2\mathbf{Dob_3}}\\
& = \frac{1}{e\cdotp 4!} =\frac{1}{34e}=\frac{1}{2(15e+2e)}=\frac{1}{2(\mathbf{Dob_4}+ \mathbf{Dob_2})}\\
& = \frac{1}{e\cdotp 5!} =\frac{1}{154e}=\frac{1}{10(15e) + 4e}=\frac{1}{10\mathbf{Dob_4} + 2\mathbf{Dob_2}}\\
& = \frac{1}{e\cdotp 6!} =\frac{1}{874e}=\frac{1}{4(203e) + 4(15e) + 2e }=\frac{1}{4\mathbf{Dob_5} + 4\mathbf{Dob_4} + \mathbf{Dob_2} }\\
& = \frac{1}{e\cdotp 7!} =\frac{1}{5914e}=\frac{1}{40(15e) + 52e}=\frac{1}{5914e}=\frac{1}{40\mathbf{Dob_4} + \mathbf{Dob_3}}\\
&\quad \quad \vdots\quad \quad \quad \quad \quad \vdots \quad \quad \quad\vdots \quad \quad \quad \quad \quad \quad \quad \quad  \quad\vdots \quad \quad
\end{align*} 
clearly  $\mathbf{!n}e^x \cdotp (\mathbf{!n} e^x)^{-1}=1.$ 
\end{proof}

\section{New equivalence to Kurepa Conjecture }\label{NewConj}
In Richard Guy's unsolved problems, number theory, section B44, the Kurepa conjecture has been listed as one of the unsolved problems. In this section we provide an equivalence to this conjecture and investigates this new equivalence. We shall make use of tools such as the greatest common divisor, the Euclidean algorithms, and many relevant approaches, full details can be found in \cite{stein1967computational, knuth2014art, sorenson1990k, sorenson1994two, williams1945numbers, fabianogcd}

\subsection{$\mathbb{F}_{n}(x)$ polynomials and $\mathbb{F}_{n}$ numbers}
The Fubini polynomial \cite{https://doi.org/10.1155/IJMMS.2005.3849} is defined as 
$$F_n(x)= \sum_{k=0}^{n}k! S(n,k)x^k$$
when $k=1$ in the Stirling numbers of the second kind $S(n,k)$ we have 
\begin{equation}\label{Fub}
    \mathbb{F}_n(x) = \sum_{k=0}^n k! S(n,1)x^k=\sum_{k=0}^n k! S(n,n)x^k.
\end{equation}

\begin{theorem}\label{FK}
Let $\sum_{k=0}^n k! S(n,1)x^k$ be as in equation \ref{Fub}, this yields the polynomial
$$\sum_{k=0}^n k! S(n,1)x^k=1+x+2x^2+\cdots+n!x^n=\mathbb{F}_{n}(x).$$   
\end{theorem}
\begin{proof}
It is well known that $S(n,1)=S(n,n)=1$ for all $n\geq 1$, where the number of blocks $k$ is fixed at $1$. It is easy to compute some few examples of this polynomial;
\begin{align}\label{PF}
\mathbb{F}_{0}(x) &= 0\\ \nonumber
\mathbb{F}_{1}(x) &= 1 + x\\  \nonumber
\mathbb{F}_{2}(x) &= 1 + x + 2x^2\\  \nonumber
\mathbb{F}_{3}(x) &= 1 + x + 2x^2 + 6x^3\\  \nonumber
\mathbb{F}_{4}(x) &= 1 + x + 2x^2 + 6x^3 + 24x^4\\  \nonumber
\mathbb{F}_{5}(x) &= 1 + x + 2x^2 + 6x^3 + 24x^4 + 120x^5\\  \nonumber
\mathbb{F}_{6}(x) &= 1 + x + 2x^2 + 6x^3 + 24x^4 + 120x^5 + 720x^6\\  \nonumber
\mathbb{F}_{7}(x) &= 1 + x + 2x^2 + 6x^3 + 24x^4 + 120x^5 + 720x^6 + 5040x^7
\end{align}
\end{proof}

\begin{definition}
    The Kurepa polynomial $\mathbb{F}_n(x)$ is defined as follows:
    \begin{equation*}
\mathbb{F}_n(x)=\sum_{k=0}^n k! S(n,1)x^k=\sum_{k=0}^n k! S(n,n)x^k=\sum_{k=0}^n k!(1)x^k=\sum_{k=0}^n k!x^k
 \end{equation*}
 \begin{equation*}
\mathbb{F}_n(x)=
\begin{cases}
	\text{0}    & \text{$n=0$;}\\
	\text{$\sum_{k=0}^n k!x^k$}   & \text{ positive integer $n\geq 2$ in the usual Kurepa factorial}.
\end{cases}
 \end{equation*}
\end{definition}

\begin{corollary}
    For $x=1$, the list of polynomials in equation \ref{PF} sums to the values of the Kurepa factorials (!n). The polynomial
    $$\mathbb{F}_{n\geq1}(x)= 1+x+2x^2+\cdots+n!x^n$$ and 
    $\mathbb{F}_{n\geq1}(1) = \sum_{k=0}^n k! S(n,1)$ yields; 
\begin{align*}\label{PF}
\mathbb{F}_{1}(x) &= 1 + x=1+1=2\\  \nonumber
\mathbb{F}_{2}(x) &= 1 + x + 2x^2=1+1+2=4\\  \nonumber
\mathbb{F}_{3}(x) &= 1 + x + 2x^2 + 6x^3=1+1+2+6=10\\  \nonumber
\mathbb{F}_{4}(x) &= 1 + x + 2x^2 + 6x^3 + 24x^4=1+1+2+6+24=34\\  \nonumber
\mathbb{F}_{5}(x) &= 1 + x + 2x^2 + 6x^3 + 24x^4 + 120x^5=1+1+2+6+24+120=154.\\  
\end{align*}  
\end{corollary}
\begin{proof}
    The proof of this is straightforward.
\end{proof}

\begin{definition}
    The Kurepa numbers $K_n$ is defined as follows
    \begin{equation*}
\mathbb{F}_n=
\begin{cases}
	\text{0}    & \text{$\mathbb{F}_{0}$;}\\
	\text{$\mathbb{F}_{n\geq1}$}   & \text{for all positive integer $n$}.
\end{cases}
 \end{equation*}
    
\end{definition}

\begin{table}[!ht]\label{T}
    \centering
    \begin{tabular}{|c|c|c|a|c|c|c|c|c|c|}\hline 
        $n$& 0 & 1 & 2 & 3 &4  &5  & 6 & 7 &8 \\\hline 
         $n!$& 1 & 1 & 2 & 6 & 24 &120  & 720 & 5040 &40320 \\\hline
        $\mathbf{!n}$ &  & 1 & 2 & 4 & 10 & 34 & 154 & 874 &5914 \\\hline
         $\mathbb{F}_{n}(x)$ & $\mathbb{F}_{0}$ &  & $\mathbb{F}_{1}$ & $\mathbb{F}_{2}$ &$\mathbb{F}_{3}$ &$\mathbb{F}_{4}$& $\mathbb{F}_{5}$ & $\mathbb{F}_{6}$ &$\mathbb{F}_{7}$ \\\hline
    \end{tabular}
    \caption{Relations between Kurepa and $\mathbb{F}_{n\geq1}$ \cite{OEIS:A003422, OEIS:A005165} }
    \label{table 5}
\end{table}

The generating function of 
 the Fubini numbers $F_n$ as given by Gross \cite{gross1962preferential} in his paper on preferential arrangement, $ \sum_n F_n(x)\dfrac{t^n}{n!}=\dfrac{1}{2-e^t}$.

Tanny \cite{tanny1975some} showed that $ \sum_n F_n(x)\dfrac{t^n}{n!}= \dfrac{1}{1-x(e^t -1)}$, he demonstrated that if $x=1$ the  $F_n$ yields an infinte series
$$F_n(x)= \dfrac{1}{2}\sum_{k=0}^{\infty}\dfrac{k^n}{2^k}. $$
The $\mathbb{F}_{n\geq1}(x)$ is different from the Fubini polynomial, this can be found in the following lemma;

\begin{lemma}\label{put}
For any integer $n=0,1,2,3,\ldots$ the $\mathbb{F}_{n\geq1}(x)\not\subset F_n(x)$, that is $$\sum_{k=0}^n k! S(n,n)x^k\not\subset \sum_{k=0}^n k! S(n,k)x^k$$ and 
the following recurrence easily holds;
$$ \mathbb{F}_{n}(x)=\mathbb{F}_{n-1}(x) +n!x^n.$$   
\end{lemma}

\begin{proof} From table \ref{table 6} below the difference between the two polynomials is trivial.
    \begin{table}[!ht]\label{6T}
    \centering
    \begin{tabular}{|c|c|c|}\hline 
        Fubini polynomials& $ \mathbb{F}_{n}(x)$ polynomials   \\\hline 
         $F_0(x) = 1$ & $\mathbb{F}_0(x) = 0$ \\\hline
        $F_1(x) = x$ & $\mathbb{F}_1(x) = 1 + x$  \\\hline
          $F_2(x) = x + 2x^2$ & $\mathbb{F}_2(x) = 1 + x + 2x^2$  \\\hline
            $F_3(x) = x + 6x^2 + 6x^3$ & $\mathbb{F}_3(x) = 1 + x + 2x^2 + 6x^3$ \\\hline
              $F_4(x) = x + 14x^2 + 36x^3 + 24x^4$ & $\mathbb{F}_4(x) = 1 + x + 2x^2 + 6x^3 + 24x^4$   \\\hline
                $F_5(x) = x + 30x^2 + 150x^3 + 240x^4 + 120x^5$ & $\mathbb{F}_5(x) = 1 + x + 2x^2 + 6x^3 + 24x^4 + 120x^5$   \\\hline
    \end{tabular}
    \caption{Relations between Kurepa and Fubini numbers(ordered Bell numbers) \cite{https://doi.org/10.1155/IJMMS.2005.3849, OEIS:A000670} }
    \label{table 6}
\end{table}

\end{proof}

\begin{lemma}\label{rt}
For any integer $n=0,1,2,3,\ldots$ the $\mathbb{F}_{n\geq1}\not\subset F_n$, that is $$\sum_{k=0}^n k! S(n,1)\not\subset \sum_{k=0}^n k! S(n,k)$$ and 
the following recurrence easily holds;
$$ \mathbb{F}_{n}=\mathbb{F}_{n-1}+n!$$   
\end{lemma}
\begin{proof}
    The proof of this lemma follows immediately from lemma \ref{put}.
\end{proof}

\begin{definition}\label{Kr}
     Let $r_n(x)$  be a polynomial defined as follows
    \begin{equation*}
r_{n}(x)=\sum_{k=0}^n \frac{k!}{2} S(n,1)x^k=\sum_{k=0}^n \frac{k!}{2} S(n,n)x^k=\sum_{k=0}^n \frac{k!}{2}x^k
 \end{equation*}
 for all non-negative integers $n.$ This polynomial satisfies the recurrence relation
 $$ r_{n}(x)=r_{n-1}(x)+ \frac{n!}{2}x^n.$$  
\end{definition}

\begin{theorem}
    Let $r_n(x)$  be the polynomial in definition \ref{Kr}, if we set $x=1$, then
    \begin{equation*}
r_{n}(1)=\sum_{k=0}^n \frac{k!}{2} S(n,1)=\sum_{k=0}^n \frac{k!}{2}
 \end{equation*}
 for all non-negative integers $n.$ This number satisfies the recurrence relation
 $$ r_{n}=r_{n-1}+ \frac{n!}{2}.$$  
\end{theorem}
\begin{proof}
    From definition \ref{Kr} and lemma \ref{rt}, the proof of this is straightforward.
\end{proof}

\begin{theorem}\label{cpr}
For any integer $n$ the following results hold
\begin{enumerate}
    \item The rational function $\dfrac{\mathbb{F}_{n}(x)}{r_n(x)}=2$ for $r_n(x)\neq 0,$
    \item For $x=1$ the number $\mathbb{F}_{n}(1)=2r_n(1)$ as in conjecture \ref{main},
     \item $\gcd(\mathbb{F}_{n},2)=2, $
    \item For $n\geq 3$ the $r_n$ is always an odd number and the
\begin{equation*}
\begin{cases}
	\text{$\gcd(r_{n},2)=1$}    & \text{coprime}\\
	\text{$\gcd(\mathbb{F}_{n}, r_{n})=r_{n}$}.
\end{cases}
 \end{equation*}
\end{enumerate} 
\end{theorem}
\begin{proof}
From definition \ref{Kr} and Theorem \ref{FK}, we know that 
 \begin{equation*}
r_{n}(x)=\sum_{k=0}^n \frac{k!}{2} S(n,1)x^k=\sum_{k=0}^n \frac{k!}{2}x^k
 \end{equation*}
 and 
 \begin{equation*}
\sum_{k=0}^n k! S(n,1)x^k=\sum_{k=0}^n k!x^k=\mathbb{F}_{n\geq 1}(x)
 \end{equation*}
the table \ref{table 6} below summarizes some few list of these polynomials;
\begin{table}[!ht]
    \centering
    \begin{tabular}{|c|c|c|c|c|a|c|c|c|c|}\hline 
        $r$ polynomials($r(x)$)& $r_n$&$\mathbb{F}_{n\geq 1}(x)$ polynomials & $\mathbb{F}_{n\geq 1}$  \\\hline 
         $r_0(x) = 0$ & $0$ &$F_0(x) = 0$& 0\\\hline
        $r_1(x) = \frac{1}{2}+ \frac{1}{2}x$ & 1&$F_1(x) = 1 + x$ & 2\\\hline
          $r_2(x) = \frac{1}{2}+ \frac{1}{2}x + x^2$ &2 & $F_2(x) = 1 + x + 2x^2$ & 4\\\hline
            $r_3(x) = \frac{1}{2}+ \frac{1}{2}x + 3x^2 + 3x^3$ &5 & $F_3(x) = 1 + x + 2x^2 + 6x^3$ & 10\\\hline
              $r_4(x) = \frac{1}{2}+ \frac{1}{2}x + 3x^2 + 3x^3 + 12x^4$ &17 & $F_4(x) = 1 + x + 2x^2 + 6x^3 + 24x^4$ & 34 \\\hline
                $r_5(x) = \frac{1}{2}+ \frac{1}{2}x + 3x^2 + 3x^3 + 12x^4+ 60x^5$ &77 & $F_5(x) = 1 + x + 2x^2 + 6x^3 + 24x^4 + 120x^5$ & 154 \\\hline
    \end{tabular}
    \caption{Relations between $r_n$ and $\mathbb{F}_{n}$ numbers \cite{OEIS:A003422} }
    \label{table 7}
\end{table}
To proof (1) we shall explicitly write repectively $r_n(x)$ and $\mathbb{F}_{n\geq 1}(x)$;
 \begin{equation*}
r_{n}(x)=\sum_{k=0}^n \frac{k!}{2}x^k=\dfrac{1}{2}\left(1+x+2x^2+\cdots+n!x^n \right)
 \end{equation*}
 and 
 \begin{equation*}
    \mathbb{F}_{n\geq1}(x)= \sum_{k=0}^n k! x^k=1+x+2x^2+\cdots+n!x^n=K_n(x)
 \end{equation*}
Let us check for some few polynomials $n\geq1$
\begin{align*}
    \dfrac{\mathbb{F}_{1}(x)}{r_1(x)}&=\dfrac{1 + x}{\frac{1}{2}+ \frac{1}{2}x}
    =\dfrac{ 1+x}{\dfrac{1}{2}\left(1+x \right)}=2, \quad \mbox{where $r_1(x) \neq 0$},
\end{align*}

\begin{align*}
    \dfrac{\mathbb{F}_{2}(x)}{r_2(x)}&=\dfrac{1+x+2x^2}{\frac{1}{2}+ \frac{1}{2}x + x^2}
    =\dfrac{ 1+x+2x^2}{\dfrac{1}{2}\left(1+x+2x^2 \right)}=2 \quad \mbox{where $r_2(x) \neq 0$}
\end{align*}
it is easy to see that for all $n\geq 1$ the rational function
\begin{align*}
    \dfrac{\mathbb{F}_{n}(x)}{r_n(x)}&=\dfrac{\sum_{k=0}^n k! x^k}{\sum_{k=0}^n \frac{k!}{2}x^k}=\dfrac{ 1+x+2x^2+\cdots+n!x^n}{\dfrac{1}{2}\left(1+x+2x^2+\cdots+n!x^n \right)}=2
\end{align*}
if the root of the polynomial $r_n(x) \neq 0,$ this completes the proof of (1).\\

From the proof (1), it is easy to check that for $x=1$ the number $\mathbb{F}_{n}(1)=2r_n(1)$, the few list of these polynomial $F_n(1)=2r_n(1)$ can be observed in table \ref{table 6} and this finishes the proof of (2). \\

Next we proof that $\gcd(\mathbb{F}_n, 2)$ is $2$, this is straight forward since we know that $\mathbb{F}_{n}=2r_n$ so we can put $\gcd(\mathbb{F}_n, 2)=\gcd(2r_n, 2)$ which is clearly $2$, thus
$$\gcd(\mathbb{F}_n, 2)=\gcd(K_n, 2)=\gcd(2\cdotp r_n, 2)=2$$
this completes the proof of ($3$).

The proof of (4): We observe from the table $6$, that $r_3=5$ which is odd(say $2t+1$), this makes the statement true for $n=3.$ Now it is easy to see that for any $k\geq 4$ the $k!$ contains at least two factors of $2$, making it even integer, so $ k!/2$ is divisible by $2$.
The sum of any number of even integers is always an even integer, thus
  \begin{equation*}
\sum_{k=4}^n \frac{k!}{2} S(n,1)x^k=\sum_{k=4}^n \frac{k!}{2}x^k=2t\quad \mbox{even number},
 \end{equation*}\\
  finally, we observe that $(2t+1=\mbox{odd})+ (2t=\mbox{even})=4t +1 (\mbox{odd number})$ \\
  explicitly $ 5+ \sum_{k=4}^n \frac{k!}{2}x^k=r_n$ thus for all $n\geq 3$ the $r_n$ is always an odd number.\\
 Finally one can realize that $\gcd(r_{n},2)=1$ since we now know that $r_n$ is odd number for all $n\geq 3$, then it is relative prime with $2.$\\
 Also, it is known that the greatest common divisor of even number and odd number is always odd
 number thus $\gcd(\mathbb{F}_{n}, r_{n})=\gcd(2\cdotp r_n, r_{n})=r_{n}$. We have proved that
\begin{equation*}
\begin{cases}
	\text{$\gcd(r_{n},2)=1$,}    & \text{coprime},\\
	\text{$\gcd(\mathbb{F}_{n}, r_{n})=r_{n}$}.
\end{cases}
 \end{equation*}
 
\end{proof}

\begin{corollary}\label{cor}
    For all nonnegative integers $n>2$, the factorial $n!$ is even.
\end{corollary}

\begin{proof}
    We shall proof the following statement  $P(n)$: $n!=n\times(n-1)\times(n-2)\times \cdots \times 3\times 2 \times1$ is even for $n>2.$
    \begin{enumerate}
        \item  $P(3)$: for any non-negative integer $n>2$,
\item $P(k)$: when $n=k$ and $k>2$ we show that $k!=2t$ for all $t\in \BZ^+,$ 
\item $P(k+1)$: when $n=k+1$ and $k>2$ we show that $(k+1)!=2T$ for all $T\in \BZ^+,$ 
 \end{enumerate}
P(3): for any non-negative integer $n>2$, let $n=3$, $3!=3\times 2 \times 1=2(3\times 1)=2(3)=6$ where $6$ is an even number.\\
P(k): Next we show for any $t$, where $t\in \BZ^+,$ that
\begin{align*}
    k!&=k\times(k-1)\times(k-2)\times \cdots \times 3\times 2 \times1\\
    &= 2\times ( k\times(k-1)\times(k-2)\times \cdots \times 3 \times1)\\
    &=2 \times t= 2t,
\end{align*}
For P(k+1): let $k$ be an integer and $t$ be as defined above,
\begin{align*}
    (k+1)!&=(k+1)k!=(k+1)[k\times(k-1)\times(k-2)\times \cdots \times 3\times 2 \times1]\\
    &= 2\times [ (k+1)k\times(k-1)\times(k-2)\times \cdots \times 3 \times1]\\
    &=2 \times t(k+1)= 2t(k+1)=2T
\end{align*}
since for all $n>2$ the statement $P(n)$ is true for $P(1), P(k)$ and $P(k+1)$ then for all  $n>2$, the factorial $n!$ is even.
\end{proof}

\begin{lemma}\label{cp}
   For all $n\geq3$ the $\gcd(r_{n},\frac{(n+1)!}{2})=1$, that is, $r_n$ and $\frac{(n+1)!}{2}$ are coprime.
\end{lemma}

\begin{proof}
In Theorem \ref{cpr} it is shown that the $\gcd(r_{n},2)=1$ for which $r_n$ and $2$ are coprime, in this lemma we want to check for all $n\geq 3$ whether $\gcd(r_{n},\frac{(n+1)!}{2})=1$. 
Let $n!=n(n-1)(n-2)\cdots 3\cdotp 2\cdotp 1$, it is known from corollary \ref{cor} and Theorem \ref{cpr} that for any $n\geq 2$ the $n!$ contains at least one factor of $2$, making it even integer, so $ (n+1)!$ is divisible by $2$, specifically, $(n+1)!=2T$ by the induction from the corollary \ref{cor}, this makes $\frac{(n+1)!}{2}=\frac{2T}{2}=T$.
Next we check for all $n\geq 3$ whether $\gcd(r_{n},\frac{n!}{2})=\gcd(r_{n},T)=1$, that is, if $r_n$ and $\frac{(n+1)!}{2}=T$ are coprime. Since it is well known that the greatest common divisor for odd number and even number is odd number, we can readily check for $\gcd(r_{n},\frac{(n+1)!}{2})=\gcd(r_{n},T)=1$, since $r_n$ is always odd number by Theorem \ref{cpr}, we can compute the greatest common divisor explicitly as below;
\begin{align*}
    \gcd \left(r_n, \frac{(n+1)!}{2}\right)&=\gcd \left(r_n, T\right)=\mbox{odd number}\\
    &=\gcd(5,12)=1\\
    &=\gcd(17,60)=1\\
    &=\gcd(77, 120)=1
\end{align*}
specifically, we known that $r_n$ is odd and $\frac{(n+1)!}{2}=T$ is even for some integer $s$, the $\gcd(r_{n},T)=1$ implies there exist some set $V=\lbrace{r_n x + T y=1|x, y\in \mbox{integers}  \rbrace}$ such that $1|r_n$ for $ r_n=q\cdotp1+\mathbb{R}$ for $\mathbb{R}=0$ and $0\leq \mathbb{R}<1.$
\begin{align*}
    r_n=q\cdotp1+\mathbb{R}\\
    \mathbb{R}&=r_n - q\cdotp1= r_n - q(r_n x_0 + Ty_0  )\\
    &=r_n -qr_n x_0 - qT y_0\\
    &=r_n(1-qx_0) + (-qT y_0),
\end{align*}
and $$V=\lbrace{\mathbb{R}=r_n x + T y=1|x=1-qx_0, y=-qy_0  \rbrace}$$ since  $0\leq \mathbb{R}<1$, with $\mathbb{R}=0$ and $1$ is the least positive integer in the set $V$, it follows that $1|r_n.$ Similarly we show $1|(\frac{n!}{2})$ and this completes the proof.
\end{proof}

\begin{theorem}\label{Equ}
For all non-zero integer $n\geq 1$, the $\gcd(\mathbb{F}_{n},(n+1)!)=2$. This is equivalent to the following; 
    \begin{enumerate}
        \item[(a)] $\gcd(2\cdotp r_{n},2T)=2\cdotp \gcd(r_{n},T)$ from lemma \ref{cp},\\ $\gcd(\mathbb{F}_{n},(n+1)!)=\gcd(2\cdotp r_{n},2T)=2\cdotp 1= 2$;
        \item[(b)]  $2=\gcd(\mathbb{F}_{n},(n+1)!)$, if and only if $\gcd\left (\frac{\mathbb{F}_{n}}{2},\frac{(n+1)!}{2}\right)=1$;
        \item[(c)] If we let $u=\mathbb{F}$ and $v=(n+1)!$, then by the results of J. Stein \cite{stein1967computational} the Binary GCD algorithm, that is,  $\gcd(u,v)=2\gcd(\frac{u}{2}, \frac{v}{2})$ where $$\gcd\left(\frac{u}{2}, \frac{v}{2}\right)=\gcd\left(\frac{\mathbb{F}_{n}}{2}, \frac{(n+1)!}{2}\right)=\gcd \left(r_n, \frac{(n+1)!}{2}\right)=1$$
        it immediately follows that $\gcd(\mathbb{F}_{n},(n+1)!)=2$.
    \end{enumerate}
\end{theorem}

\begin{proof}
The statement of the Theorem for all non-zero integer $n\geq 1$, the $\gcd(\mathbb{F}_{n},\frac{(n+1)!}{2})=2$ has many equivalence and to show the detailed proof it is prudent to work out the equivalence for clarity. 
To prove $(a)$, Let 
\begin{align*}
    \gcd(\mathbb{F}_{n},(n+1)!)&=\gcd(2\cdotp r_{n},2T)\\
    &=2\cdotp \gcd(r_n, T) \\
    &=2 \cdotp 1= 2 \quad \mbox{since $\gcd(r_n, T)=1$ from lemma \ref{cp}}
\end{align*}
thus $\mathbb{F}_{n} x + (n+1)! y=2$ is obvious using an important and well known property, that is, if $d>0,$ then $$\gcd(d\cdotp a,d\cdotp b)=d\cdotp \gcd(a, b)$$ for details on this see \cite{knuth2014art}.

To prove $(b)$ If $\mathbb{F}_{n} x + (n+1)! y=2$, then $$ \dfrac{\mathbb{F}_{n}}{2} x + \dfrac{(n+1)!}{2} y=1$$ this implies that
\begin{align*}
    &\dfrac{\mathbb{F}_{n}}{2} x + \dfrac{(n+1)!}{2} y=1\\
    &\dfrac{2\cdotp r_{n}}{2} x + \dfrac{2T}{2}y =\dfrac{2}{2} \quad \mbox{but}\\
    &r_n x + T y=1 \implies \gcd(r_{n},\frac{(n+1)!}{2})=1 \quad \mbox{is coprime as in lemma \ref{cp}}
\end{align*}
specifically,
\begin{align*}
    &\gcd\left(\frac{\mathbb{F}_{n}}{2}, \frac{(n+1)!}{2}\right)=1 \iff\\
    &2\gcd\left(\frac{\mathbb{F}_{n}}{2}, \frac{(n+1)!}{2}\right)=2 \iff \\
    &\gcd\left(\frac{2\cdotp \mathbb{F}_{n}}{2}, \frac{2 \cdotp (n+1)!}{2}\right)=2 \iff \\
    &\gcd(\mathbb{F}_{n},(n+1)!)=2 \\
\end{align*}
this completes the proof of $(b).$
The proof of $(c)$ relies heavily on the following properties \cite{stein1967computational, knuth2014art} for integers $u$ and $v$;
\begin{enumerate}
    \item  If $u$ and $v$ are both even, then $\gcd(u,v)=2\gcd\left(\frac{u}{2}, \frac{v}{2}\right)$
    \item  If $u$ is even and $v$ is odd, then $\gcd(u,v)=\gcd\left(\frac{u}{2}, v\right)$
    \item  If $u$ and $v$ are both odd, then $\gcd(u,v)=\gcd\left(\frac{|u-v|}{2}, v\right)$
    \item  $\gcd(0,v)=v$ and $\gcd(u,0)=u$.
\end{enumerate}
Now to compute the $\gcd(\mathbb{F}_{n},(n+1)!)$ using the binary GCD algorithm for $n\geq 1;$ \\
the first step of the binary GCD algorithm is to extract common factors of $2$,\\
after extracting the initial factor of $2$, then we proceeds with finding the $\gcd\left(\frac{\mathbb{F}_{n}}{2} , \frac{(n+1)!}{2}\right)$;
Consider $n=2$ and let $u=\mathbb{F}_2$ and $v=3!$, then by the properties of Binary GCD algorithm above, we observe that both $ \mathbb{F}_2$ and $3!$ are even so 
\begin{align*}
    2\gcd\left(\frac{u}{2}, \frac{v}{2}\right)&= 2\gcd\left(\frac{\mathbb{F}_{2}}{2}, \frac{3!}{2}\right)=2 \gcd\left(\frac{4}{2}, \frac{6}{2}\right)\\
    &=2\gcd \left(r_2, \frac{6}{2}\right)=2\gcd(2,3)
\end{align*}
now $u=2$ and $v=1$, from property (2) we make $u$ odd, thus the $\gcd(\frac{2}{2},3)=\gcd(1,3)$. We proceed with the algorithm by using property $(3)$ until we obtain $\gcd(0,1)=1.$ Thus 
$$\gcd \left(r_n, \frac{(n+1)!}{2}\right)=2\gcd(2,3)=2\cdotp 1=2.$$
Next we consider $n=3$ and note that both $ \mathbb{F}_3$ and $4!$ are even so 
\begin{align*}
     2\gcd\left(\frac{\mathbb{F}_{3}}{2}, \frac{4!}{2}\right)&=2 \gcd\left(\frac{10}{2}, \frac{24}{2}\right)\\
    &=2\gcd \left(r_3, \frac{24}{2}\right)=2\gcd(5,12)
\end{align*}
we divide $12$ by $2$ until we have 
$u=5$ and $v=3$, from property (3) we see both are odd, thus the $\gcd\left(\frac{|5-3|}{2}, 3\right)=\gcd(1,3)$. We proceed with the algorithm until we obtain $\gcd(1,1)=1.$
this holds for all integers $k\geq 1$, that is
\begin{align*}
    2\gcd\left(\frac{u}{2}, \frac{v}{2}\right)&= 2\gcd\left(\frac{\mathbb{F}_{k}}{2}, \frac{(k+1)!}{2}\right)\\
    &=2\gcd \left(r_k, \frac{(k+1)!}{2}\right) \quad \mbox{proceeding with the algorithm}\\
    &=2\cdotp 1=2 \quad \mbox{from lemma $\ref{cpr}$}
\end{align*}
this completes the proof for $(c).$
One observes that the condition $\gcd(u,v) = 2 \cdotp \gcd\left(\frac{u}{2}, \frac{v}{2}\right)$ is the initial and most crucial step in elucidating why the GCD of the $\mathbb{F}_{n}$ and the factorial $(n+1)!$ equals $2$. The problem is established for the subsequent phases of the binary GCD, where the oddity of $r_n$ and the parity of $\frac{(n+1)!}{2}$ for $n\geq1$ ultimately results in a GCD of $1$, the table below gives a clear view, thus for all non-zero integer $n\geq 1$, the $\gcd(\mathbb{F}_{n},(n+1)!)=2$. There are other extensions and accelerated forms of the Binary GCD algorithm proposed by J.Sorenson and many others \cite{brent1985systolic, sorenson1990k, sorenson1994two}

\begin{tabular}{ccccccc}
\toprule
$n$ & $n!$ & $(n+1)!/2=T$ & $\mathbb{F}_{n}$ &$r_n$ & $\gcd(r_{n},T)$& Binary $\gcd(u=\mathbb{F}, v=(n+1)!)$\\
\midrule
0 & 1 & $\frac{1}{2}$ &  &  & & \\
1 & 1 & $1$ & 2 &1 & 1& \\
2 & 2 & 3 & 4 &2 & 1&2\\
3 & 6 & 12 & 10 &5  & 1&2\\
4 & 24 & 60 & 34 &17  & 1& 2\\
5 & 120 & 360 &154  &77 & 1& 2\\
6 & 720 & 2,520 & 874 &437& 1& 2\\
7 & 5,040 & 20,160 & 5,914 &2,957& 1& 2\\
8 & 40,320 & 181,440 & 46,234 &23,117& 1& 2\\
9 & 362,880 & 1,814,400 & 409,114 &204,557& 1& 2\\
10 & 3,628,800 & $1.99584
\times10^7$ & 4,037,914 &2,018,957& 1 & 2\\
\bottomrule
\end{tabular}
\end{proof}

\begin{theorem}\label{final}
Let $K_n=\mathbf{!n}$ be the Kurepa factorial defined by
     $$K_n=\mathbf{!n}=\sum_{m=0}^{n-1}m!= 0! + 1!+ 2! + 3! + 4!+ 5!  + \cdots +(n-1)!=S_0(n).$$ with $m<n$,
and let $\mathbb{F}_n=\sum_{k=0}^n k!S(n,1)$ as already defined, 
then the Kurepa Conjecture \ref{main} is equivalent to Theorem \ref{Equ}.
\end{theorem}
\begin{proof}
Let $K_n=\mathbf{!n}=\sum_{m=1}^{n-1}m!= 0! + 1!+ 2! + 3! + 4!+ 5!  + \cdots +(n-1)!$, and let 
 $K_{n+1}=\sum_{m=0}^{n-1}m!= 0! + 1!+ 2! + 3! + 4!+ 5!  + \cdots +(n-1)!+ n!$, then
 \begin{align*}
     K_{n+1}-K_n&= (0! + 1!+ 2! + 3! + 4!+ 5!  + \cdots +(n-1)!+n!)\\
     &= -(0! + 1!+ 2! + 3! + 4!+ 5!  + \cdots +(n-1)!)=n!
 \end{align*}
thus $ K_{n+1}-K_n=n!$ and the Kurepa conjecture \ref{main}, that is,
$$\gcd(!n, n!)=\gcd(K_n, n!)=\gcd(K_n, (K_{n+1}-K_n)) $$
for $n\geq 2,$ the $\gcd(K_n, K_{n+1})=2$, also the $\gcd(K_n, K_{n})=2$. 
Now since $\mathbb{F}_n= K_n $ (see table \ref{table 5}) we have that;
\begin{align*}
    &\gcd(\mathbb{F}_{n},(n+1)!)\sim \gcd(!n, n!)=\gcd(K_n, n!),\\
    &\gcd(\mathbb{F}_{n}, (\mathbb{F}_{n+1}-\mathbb{F}_n))\sim\gcd(K_n, (K_{n+1}-K_n)). 
\end{align*}
From Theorem \ref{Equ} the greatest common divisor between $\mathbb{F}_{n}$ and $(n+1)!$, that is,\\ $\gcd(\mathbb{F}_{n},(n+1)!)=2$ thus $\gcd(!n, n!)=\gcd(K_n, n!)=\gcd(K_n, (K_{n+1}-K_n))=2.$
This completes the proof.
\end{proof}

\begin{theorem}\label{induction}
   If $\gcd(!n, n!)=2$ then by induction the greatest common divisor of $(n+1)$ for the left factorial and the right factorial is also 2, that is, $$\gcd ( K_{n+1}, (K_{n+2}- K_{n+1}))=\gcd (!(n+1), (n+1)!)=2.$$
\end{theorem}
\begin{proof}
    Given that $\gcd(!n,n!)=2$ for $n\geq2$, we can verify that $\gcd (!(n+1), (n+1)!)=2$.
    It is known that the Kurepa factorial $\mathbf{!n}$ satisfies the recurrence formula $!(n+1)=\mathbf{!n}+n!$, then there exists an integer $r\geq 2$ such that $\mathbf{!n}+n!=2r$, and from Lemma \ref{induct}, the $2|!(n+1)$ and also $2|\mathbf{!n}+n!$
    \begin{align*}
        \gcd (!(n+1), (n+1)!)&=2\gcd \left(\frac{!(n+1)}{2}, \frac{ (n+1)!}{2}\right)\\
        &=2\gcd \left(\frac{\mathbf{!n}+n!}{2}, T\right)\\
        &=2\gcd \left(r, T\right)= 2\cdotp 1=2 
    \end{align*}
from table $1$, it can be observed that all $r>2$ is odd(also check table 6 for $r_n$ values which is same as the $r$ values) and from corollary \ref{cor}, $T$ is always even.
Also, from Theorem \ref{Equ} and Theorem \ref{final}, we can show that $2=\gcd (!(n+1), (n+1)!)$, if and only if $\gcd \left(\frac{!(n+1)}{2}, \frac{ (n+1)!}{2}\right)=1$, 
\begin{align*}
    \gcd \left(\frac{!(n+1)}{2}, \frac{ (n+1)!}{2}\right)&=1 \iff\\
    2\gcd \left(\frac{!(n+1)}{2}, \frac{ (n+1)!}{2}\right)&=2 \iff \\
    \gcd \left(\frac{2\cdotp !(n+1)}{2}, \frac{2 \cdotp (n+1)!}{2}\right)&=2 \iff \\
    \gcd (!(n+1), (n+1)!)&=2. \\
\end{align*}

\end{proof}

\begin{corollary}\label{gcd}
Given that $G_n=\gcd(!n, n!)=2$ and $G_{n+1}=\gcd [!(n+1), (n+1)!]=2$, the following results hold:
\begin{enumerate}
    \item The greatest common divisor of $G_n$ and $G_{n+1}$ is always 2, that is, $$\mathbb{M}_n=\gcd(G_n, G_{n+1})=2;$$
    \item the inequalities $G_1\leq G_2\leq G_3\leq \cdots \leq G_n\cdots\leq G_{n+1}\leq 2$ is an increasing sequence and $G_{n+1}\leq 2$ bounded above by $2$.
\end{enumerate}
\end{corollary}
\begin{proof}
    The proof of this corollary is trivial.
\end{proof}

\subsection{Shifted Alternating $\mathbb{F}_n$ number}
\begin{theorem}\label{FK}
Let $\sum_{k=0}^n k! S(n,1)(-x)^k$ be the reciprocal of equation \ref{Fub}, this yields the polynomial
$$\sum_{k=0}^n k! S(n,1)(-x)^k=1-x+2x^2-6x^3+\cdots+(-1)^nn!x^n=\mathbb{F}_{n}(-x).$$   
\end{theorem}
\begin{proof}
For $S(n,1)=S(n,n)=1$ for all $n\geq 1$, where the number of blocks $k$ is fixed at $1$. One can compute some few examples of this polynomial;
\begin{align}\label{FT}
\mathbb{F}_{0}(-x) &= 0\\ \nonumber
\mathbb{F}_{1}(-x) &= 1 - x\\  \nonumber
\mathbb{F}_{2}(-x) &= 1 - x + 2x^2\\  \nonumber
\mathbb{F}_{3}(-x) &= 1 - x + 2x^2 - 6x^3\\  \nonumber
\mathbb{F}_{4}(-x) &= 1 - x + 2x^2 - 6x^3 + 24x^4\\  \nonumber
\mathbb{F}_{5}(-x) &= 1 - x + 2x^2 - 6x^3 + 24x^4 - 120x^5\\  \nonumber
\end{align}
\end{proof}

\begin{definition}
    The shifted alternating $\mathbb{F}_n(-x)$ is defined as follows:
    \begin{equation*}
\mathbb{F}_n(-x)=\sum_{k=0}^n k! S(n,1)(-x)^k=\sum_{k=0}^n k!(-x)^k
 \end{equation*}
 as shown in the table below:
 \begin{equation*}
\mathbb{F}_n(-x)=
\begin{cases}
	\text{0}    & \text{$n=0$;}\\
	\text{$\sum_{k=0}^n k!(-x)^k$}   & \text{ positive integer $n\geq 1$ is the usual  factorial}.
\end{cases}
 \end{equation*}
\end{definition}

\begin{corollary}
    For $x=1$, the list of polynomials in equation \ref{FT} sums to the values of the factorials $A_n^s$. The polynomial
    $$\mathbb{F}_{n}(-x)= 1-x+2x^2-6x^3+\cdots+(-1)^nn!x^n$$ and 
    $\mathbb{F}_{n\geq1}(-1) = \sum_{k=0}^n (-1)^k k! $ yields; 
\begin{align*}\label{PF}
\mathbb{F}_{1}(-x) &= 1 - x=1-1=0\\  \nonumber
\mathbb{F}_{2}(-x) &= 1 - x + 2x^2=1-1+2=2\\  \nonumber
\mathbb{F}_{3}(-x) &= 1 - x + 2x^2 - 6x^3=1-1+2-6=4\\  \nonumber
\mathbb{F}_{4}(-x) &= 1 - x + 2x^2 - 6x^3 + 24x^4=1-1+2-6+24=20\\  \nonumber
\mathbb{F}_{5}(-x) &= 1 - x + 2x^2 - 6x^3 + 24x^4 - 120x^5=1-1+2-6+24-120=-100.\\  
\end{align*}  
\end{corollary}
\begin{proof}
    The proof of this is straightforward.
\end{proof}

\begin{definition}
    The Kurepa numbers $K_n$ is defined as follows
    \begin{equation*}
\mathbb{F}_n(-1)=
\begin{cases}
	 \text{$\mathbb{F}_{0}$,}  & $\text{n=0}$ \\
	\text{$\mathbb{F}_{n\geq1}(-1)$}   & \text{for all positive integer $n$}.
\end{cases}
 \end{equation*}
    
\end{definition}

\begin{table}[!ht]\label{T}
    \centering
    \begin{tabular}{|c|c|c|a|c|c|c|c|c|c|}\hline 
        $n$& 0 & 1 & 2 & 3 &4  &5  & 6 & 7 &8 \\\hline 
        $\mathbb{F}_{n}(1)$ & 1 & 2 & 4 & 10 & 34 & 154 & 874 & 5914 &46234 \\\hline
         $\mathbb{F}_{n}(-1)$ & 0 &1  & 0 & 2 &-4 &20& -100 & 620 &-4420 \\\hline
    \end{tabular}
    \caption{Relations between $\mathbb{F}_{n}(-1)$ and $\mathbb{F}_{n}(1)$ numbers  }
    \label{table 8}
\end{table}

\begin{theorem}
For all nonnegative integer $n$, the following results hold;
  \begin{enumerate}
      \item [(i)] The greatest common divisor, $\mathbb{H}_n=\gcd( |A_{n}^s|, K_n)=2$ for all $n>2.$
      \item[(ii)] The greatest common divisor, $\mathbb{H}_{n+1}=\gcd( |A_{n+1}^s|, K_{n+1})=2$ for all $n>2.$
      \item [(iii)] The greatest common divisor of $\mathbb{H}_n$ and $\mathbb{H}_{n+1}$ is always 2, that is, $$\mathbb{W}_n=\gcd(\mathbb{H}_{n}, \mathbb{H}_{n+1})=2.$$
    \item [(iv)] the inequalities $\mathbb{H}_1\leq \mathbb{H}_2\leq \mathbb{H}_3\leq \cdots \leq \mathbb{H}_n\cdots\leq \mathbb{H}_{n+1}\leq 2$ is an increasing sequence and $\mathbb{H}_{n+1}\leq 2$ bounded above by $2$.
  \end{enumerate}  
\end{theorem}
\begin{proof}
    The proof of this is trivial when one sees that $K_n$ is divisible by $2$ and $A_{n}^s$ can be divisible by $2$ or a higher even number. specifically, $$|A_{n}^s|=|\mathbb{F}_{n}(-1)|= \sum_{k=0}^n k!(-x)^{n-k}.$$ A few sketches can be seen in the table \ref{table 9} below and following the techniques from theorem \ref{Equ} the result is immediate.
\end{proof}

\begin{table}[!ht]\label{other}
\centering
\begin{tabular}{|c|c|c|c|c|}\hline
$n$ & $A_n^s = \sum_{m=0}^{n-1} (-1)^m m!$ & $K_n =!n= \sum_{m=0}^{n-1} m!$ & $\gcd( |A_{n}^s|, K_n)$ &$\gcd( |A_{n+1}^s|, K_{n+1})$\\
\hline
0 & 0 & 0 & N/A& N/A\\
1 & 1 & 1 & 1 & 1\\ \rowcolor{LightCyan}
2 & 0 & 2 & 2 & 2\\ 
3 & 2 & 4 & 2&2 \\
4 & -4 & 10 & 2& 2\\
5 & 20 & 34 & 2& 2\\
6 & -100 & 154& 2& 2\\
7 & 620 & 874& 2& 2 \\
8 & -4420 & 5914 & 2& 2 \\
9 & 35900 & 46234& 2 & 2\\
10 & -326980 & 409114& 2 & 2\\
\hline
\end{tabular}
  \caption{The $\gcd$ of Kurepa and shifted alternating sum of factorials}
    \label{table 9}
\end{table}

\begin{theorem}
    The greatest common divisor, $\gcd(\mathbb{M}_{n}, \mathbb{W}_{n})=2$ for all nonnegative integer $n> 2.$
\end{theorem}
\begin{proof}
    The proof of this is trivial.
\end{proof}


\subsection{Altered $\mathbb{F}_{n}$ Sequence(Altered Kurepa Sequence)}

The values of the series $\mathbb{F}_{n}$ are known to be the Kurepa factorials $K_n(\mathbf{!n})$ for $n\geq1$ and $n>2$, respectively. Since it addresses one of the most important mathematical problems, this new insight is not only a coincidence.
Kurepa Conjecture \ref{main} is identical to Theorem \ref{Equ}, whereas Theorem \ref{induction} explains their inductive step. This subsection deals with various altered(shifted) $\mathbb{F}_{n}$ sequences and examines how their greatest common divisors behave. We shall give some Theorems, lemma, and then propose some open problem and conjecture.

\begin{table}[!ht]\label{T}
    \centering
    \begin{tabular}{|c|c|c|a|c|c|c|c|c|c|}\hline 
        $n$& 0 & 1 & 2 & 3 &4  &5  & 6 & 7 &8 \\\hline 
         $\mathbb{F}_n$& 0 & 2 & 4 & 10 & 34 &154  & 874 & 5914 & 46234 \\\hline
          $\mathbb{F}_{2n}$& 0 & 4 & 34 & 874 & 46234 &4037914  & 522956314 & 93928268314 & 22324392524314 \\\hline   
        $\mathbb{F}_{n+1}$ & 2 & 4 & 10 & 34 & 154 & 874 & 5914 & 46234 &409114 \\\hline
         $\mathbb{F}_{n+2}$ & 4 & 10 & 34 & 154 & 874 & 5914 & 46234 & 409114 &4037914 \\\hline
          $(n+2)!=\mathbb{F}_{n+2}-\mathbb{F}_{n+1}$ & 2 & 6 & 24 & 120 & 720 & 5040 & 40320 & 362880 &3628800 \\\hline
           $(n+1)!=\mathbb{F}_{n+1}-\mathbb{F}_{n}$ & 1 & 2 & 6 & 24 & 120 & 720 & 5040 & 40320 &362880 \\\hline
         $\mathbb{F}_{n}-1$ & NA & 1 & 3 & 9 & 33 & 153 & 873 & 5913 &46233 \\\hline
         $\mathbb{F}_{n}+1$ &1  & 3 & 5 & 11 & 35 & 155 & 875 & 5915 &46235 \\\hline
         $\mathbb{A}_n=\mathbb{F}_{n}+(-1)^n$ & 1 & 1 & 5 & 9 & 35 & 153 & 875 & 5913 & 46235\\\hline
         $\mathbb{B}_n=\mathbb{F}_{n}-(-1)^n$ & NA & 3 & 3 & 11 & 33 & 155 & 873 & 5915 &46233 \\\hline
         $\mathbb{F}_{n}+2$ & 2 & 4 & 6 & 12 & 36 & 156 & 876 & 5916 &46236 \\\hline
        $\mathbb{F}_{n}-2$ & -2 & 0 & 2 & 8 & 32 & 152 & 872 & 5912 &46232 \\\hline
         $\mathbb{F}_{n+1}+1$ & 2 & 3 & 5 & 11 & 35 & 155 & 875 & 5915 & 46235 \\\hline
        $\mathbb{F}_{n+1}-1$ & 1 & 3 & 9 & 33 & 153 & 873 & 5913 & 46233 &409113 \\\hline
        $\mathbb{A}_{n+1}=\mathbb{F}_{n+1}+(-1)^{n+1}$ & 1 & 5 & 9 & 35 & 153 & 875 & 5913 & 46235 &409113 \\\hline
        $\mathbb{B}_{n+1}=\mathbb{F}_{n+1}-(-1)^{n+1}$ & 3 & 3 & 11 & 33 & 155 & 873 & 5915 & 46233 &409115 \\\hline
        $\mathbb{F}_{n+1}+2$ & 4 & 6 & 12 & 36 & 156 & 876 & 5916 & 46236 &409116 \\\hline
        $\mathbb{F}_{n+1}-2$ & 0 & 2 & 8 & 32 & 152 & 872 & 5912 & 46232 &409112 \\\hline\rowcolor{LightCyan}
        $\mathbb{F}_{n}+a$ & * & * & * & * & * & * & * & * &* \\\hline  \rowcolor{LightCyan}
        $\mathbb{F}_{n+1}+a$ & * & * & * & * & * & * & * & * & * \\\hline

    \end{tabular}
    \caption{Some altered $\mathbb{F}_{n}$ sequences}
    \label{table 10}
\end{table}

\begin{theorem}
Let $\mathbb{F}_n=\sum_{k=0}^n k!S(n,1)$, then for all nonegative integers $n$ and $r$ the following results hold:
\begin{enumerate}
    \item  $\gcd(\mathbb{F}_{n+1}, \mathbb{F}_n)=2,$
    \item $\gcd(\mathbb{F}_{n+2}, \mathbb{F}_{n+1})=2,$
    \item $\gcd(\mathbb{F}_{n}, (\mathbb{F}_{n+1}-\mathbb{F}_n))=\gcd(\mathbb{F}_{n}, (n+1)!)=2,$
    \item $\gcd(\mathbb{F}_{n}, \mathbb{F}_{n+1},\cdots, \mathbb{F}_{n+r}, \cdots )=2, \quad r>0.$
\end{enumerate}
\end{theorem}
\begin{proof}
    Using Theorem \ref{Equ}, the prove of this Theorem is straightforward.
\end{proof}

\begin{theorem}
    For any integer $n>0$ the $\gcd(\mathbb{F}_n +a,  \mathbb{F}_{n+1}+a)=\mathcal{F}_n(a)$ where $a$ is any constant.
\end{theorem}

\begin{lemma}
    For any nonnegative integer $n$ the $\gcd(\mathbb{F}_n +2,  \mathbb{F}_{n+1}+2)=\mathcal{F}_n(2)$ where;
\begin{equation*}\mathcal{F}_n(2)=
\begin{cases}
	\text{$1$},    & \text{if $n=0$, }\\
	\text{$2$},  & \text{if $n=1$, }\\
    \text{$6$},  & \text{if $n=6$, }\\
     \text{$12$},  & \text{otherwise. }
\end{cases}
 \end{equation*}
\end{lemma}

\begin{lemma}
    For any nonnegative integer $n$ the $\gcd(\mathbb{F}_n +3,  \mathbb{F}_{n+1}+3)=\mathcal{F}_n(3)$ where;
\begin{equation*}\mathcal{F}_n(3)=
\begin{cases}
	\text{$1$},    & \text{if $0\leq n<11$, }\\
	\text{$13$},  & \text{if $n\geq11 $. }
\end{cases}
 \end{equation*}
    
\end{lemma}

\begin{lemma}
    For any nonnegative integer $n$ the $\gcd(\mathbb{F}_n +4,  \mathbb{F}_{n+1}+4)=\mathcal{F}_n(4)$ where;
\begin{equation*}\mathcal{F}_n(4)=
\begin{cases}
	\text{$1$},    & \text{if $n=0$, }\\
	\text{$2$},  & \text{otherwise. }
\end{cases}
 \end{equation*}
\end{lemma}

\begin{lemma}
    For any nonnegative integer $n$ the $\gcd(\mathbb{F}_n +5,  \mathbb{F}_{n+1}+5)=\mathcal{F}_n(5)$ where;
\begin{equation*}\mathcal{F}_n(5)=
\begin{cases}
	\text{$1$},    & \text{if $n=0$ or $n=1$, }\\
	\text{$3$},  & \text{otherwise. }
\end{cases}
 \end{equation*}
    
\end{lemma}

\begin{theorem}
    For any integer $n>0$ the $\gcd(\mathbb{F}_n +(\mathbf{a})^n,  \mathbb{F}_{n+1}+(\mathbf{a})^{n+1})=\mathcal{F}_n(\mathbf{a})$ where $\mathbf{a}$ is any constant.
\end{theorem}

\begin{theorem}
    For all nonnegative integers $n$ the $\gcd(\mathbb{A}_n,  \mathbb{A}_{n+1})=1. $
\end{theorem}

\begin{theorem}
    Given that $\mathbb{B}_n=\mathbb{F}_{n}-(-1)^n $ then for nonnegative integers $n$ the
\begin{equation*}\gcd(\mathbb{B}_n,  \mathbb{B}_{n+1})=
\begin{cases}
	\text{$1$},    & \text{if $n=2$ and $n\geq4$, }\\
	\text{$3$},  & \text{if $n=0$ and $n=1$, }\\
    \text{$11$},  & \text{otherwise. }
\end{cases}
 \end{equation*}

\end{theorem}

\begin{lemma}
 For all nonnegative integers $n$, the following results holds:

\begin{equation*}\gcd((\mathbb{F}_{n}+(-1)^n),(\mathbb{F}_{n}-(-1)^n))=
\begin{cases}
	\text{$2$},    & \text{if $n=0$,}\\
	\text{$1$},  & \text{if $n\geq 1$, }
\end{cases}
 \end{equation*}
and 
 \begin{equation*}\gcd((\mathbb{F}_{n}+1)),(\mathbb{F}_{n}-1)))=
\begin{cases}
	\text{$2$},    & \text{if $n=0$,}\\
	\text{$1$},  & \text{if $n\geq 1$. }
\end{cases}
 \end{equation*}
 
\end{lemma}

\begin{conjecture}\label{ATTA}
 For every Kurepa factorial$(K_n)$ or $\mathbb{F}_{n}$ sequence, the greatest common divisor between all successive members of the of $\mathbb{F}_{n}$ or $K_n$, that is,
 $$\gcd(K_n +a,  K_{n+1}+a)=\gcd(\mathbb{F}_n +a,  \mathbb{F}_{n+1}+a)=\mathcal{F}_n(a) $$
 is bounded above by $2$, specifically;
 \begin{equation*}\mathcal{F}_n(a)=
\begin{cases}
	\text{$\mathcal{F}_n(0)$},    & \text{if $n>1$, }\\
	\text{$\mathcal{F}_n(4)$},  & \text{if $n\geq 1$, }
\end{cases}
 \end{equation*}
 the values $\mathcal{F}_n(0)$, $\mathcal{F}_n(4)$ as defined above are bounded above by $2.$
 The question is to find all values for which $\mathcal{F}_n(a)$ is bounded above by $2.$
\end{conjecture}


\section{Logarithm, Natural logs of Kurepa Sequence}\label{Entrolog}
In this section we consider the general properties of the logarithm of Bell numbers, Dobinski numbers and then extend it to that of Kurepa sequence. We shall also investigate the natural logarithm and the log of base $2$ as well as base $10$ of these numbers.
From definitions \ref{Dobi} and \ref{Bell} it is observed that the identity
$$e\mathbf{Bell_n}=\mathbf{Dob_n}$$ means $$\mathbf{Bell_n}=\dfrac{\mathbf{Dob_n}}{e},$$
now taking log base on both sides yield the following:
\begin{align*}
    \log_{base}\mathbf{Bell_n}&= \log_{base}\left[\dfrac{\mathbf{Dob_n}}{e}\right]\\ \nonumber
    &=\log_{base}(\mathbf{Dob_n})-\log_{base}(e)
\end{align*}
If we set $base=\exp{(1)}=e$, we have the natural Log

\begin{align}\label{entropy}
    \log_{e}\mathbf{Bell_n}&= \log_{e}\left[\dfrac{\mathbf{Dob_n}}{e}\right]\\ \nonumber
   \ln{\mathbf{Bell_n}} &=\ln(\mathbf{Dob_n})-\ln(e)\\ \nonumber
   &=\ln(\mathbf{Dob_n})-1.
\end{align}

\begin{lemma}\label{lnkure}
    The sum of natural $\log$ of $\mathbf{Dob_n}$ is the natural log of the Kurepa sequence plus non-negative integer $n$, that is,
\begin{align*}
    \ln \lbrace K_n\rbrace_{n\geq 1}+ n&=\sum_{i=1}^{n}\mathbf{\Phi}_i \ln\mathbf{Dob}_i.
\end{align*}
\end{lemma}
\begin{proof}The proof of this is straight forward,
\begin{align*}
   \sum_{i=1}^{n}\mathbf{\Phi}_i \ln\mathbf{Dob}_i&= \ln (\lbrace K_n\rbrace_{n\geq 1}\cdotp e^n)\\
   &= \ln \lbrace K_n\rbrace_{n\geq 1} + \ln(e^n)\\
    &=\ln \lbrace K_n\rbrace_{n\geq 1}+ n.
\end{align*}

\end{proof}

\begin{theorem}\label{lin}
    The natural log of the Kurepa sequence $$ \ln \lbrace K_n\rbrace_{n\geq 1}=\mathcal{Q}+\sum_{i=1}^{n}\mathbf{\Phi}_i \ln\mathbf{Bell}_i=\mathcal{Q}+\sum_{i=1}^{n}\mathbf{\Phi}_i\left[ \ln \sum_{k\geq 1}^{n} S(n,k)\right],$$ where $S(n,k)$ is Stirling numbers of the second kind, $\mathbf{\Phi}_i$ are constant coefficients and $\mathcal{Q}=\ln(Constant).$
\end{theorem}
\begin{proof}
\begin{align*}
 \ln[\lbrace K_n\rbrace_{n\geq 1}\cdotp e^n]&=\ln(K_1\cdotp e) + \ln(K_2 \cdotp e) + \ln(K_3 \cdotp e )+ \ln(K_4 \cdotp e) +\cdots + \ln (K_n \cdotp e)\\ 
&= !1\cdotp e +!2\cdotp e + !3\cdotp e + !4\cdotp e + \cdots + \mathbf{!n}\cdotp e\\ \mbox{where} \quad
&\ln{ (K_1e)}= \ln{(!1e)} = \ln{(1\cdotp e)} = \ln \mathbf{Dob_1}=\ln \mathbf{Dob_0}\\
&\ln{ (K_1)}+ \ln{ e}=\ln \mathbf{Dob_1}\\
&\ln{ (K_1)}= \ln \mathbf{Dob_1}-\ln{ e}= \ln \mathbf{Dob_1}-1= \ln{\mathbf{Bell_1}}\\
&\mbox{similarly we can compute: }\\
&\ln{ (K_2)}= \ln \mathbf{Dob_2}-1= \ln{\mathbf{Bell_2}}\\
&\ln{ (K_3)}= \ln{(2 \mathbf{Dob_2})}-1= \ln{2} + \ln{\mathbf{Bell_2}}\\
&\ln{ (K_4)}= \ln{(2 \mathbf{Dob_3})}-1= \ln{2} + \ln{\mathbf{Bell_3}}\\
&\ln{ (K_5)}= \ln{2 (\mathbf{Dob_4}+ \mathbf{Dob_2}) }-1= \ln{2} + \ln{\mathbf{Bell_4}}+ \ln{\mathbf{Bell_2}} \\
&\vdots \quad \quad \quad \quad \vdots \quad \quad \quad\vdots \quad \quad \quad \quad \quad \quad \quad\vdots \quad \quad
\end{align*}
If $n=5$ and uisng lemma \ref{lnkure} the sequence becomes
\begin{align*}
     \ln[\lbrace K_5\rbrace_{n\geq 1}\cdotp e^5]&=\ln(K_1\cdotp e) + \ln(K_2 \cdotp e) + \ln(K_3 \cdotp e )+ \ln(K_4 \cdotp e)\\ 
&= !1\cdotp e +!2\cdotp e + !3\cdotp e + !4\cdotp e + \cdots + \mathbf{!n}\cdotp e\\ 
 \ln \lbrace K_5\rbrace_{n\geq 1} + \ln e^5 &= 1 e + 2e + 4e + 10e+34e+ \cdots\\
 \ln \lbrace K_5\rbrace_{n\geq 1} + 5\ln e &= \ln \mathbf{Dob_1}+ \ln \mathbf{Dob_2}+  \ln[2\mathbf{Dob_2}] + \ln[2\mathbf{Dob_3}] \\
 &+  \ln[2(\mathbf{Dob_4}+ \mathbf{Dob_2})]\\
 \ln \lbrace K_5\rbrace_{n\geq 1} + 5 &= \ln \mathbf{Dob_1}+ \ln \mathbf{Dob_2}+  \ln[2\mathbf{Dob_2}] + \ln[2\mathbf{Dob_3}] \\
 &+  \ln[2(\mathbf{Dob_4}+ \mathbf{Dob_2})]\\
 \ln \lbrace K_5\rbrace_{n\geq 1} &= \ln \mathbf{Dob_1}+ \ln \mathbf{Dob_2}+  \ln[2\mathbf{Dob_2}] + \ln[2\mathbf{Dob_3}] \\
 &+  \ln[2(\mathbf{Dob_4}+ \mathbf{Dob_2})]-5\\
 \ln \lbrace K_5\rbrace_{n\geq 1}&= \ln \mathbf{Dob_1}-1+ \ln \mathbf{Dob_2}-1+  \ln[2\mathbf{Dob_2}]-1 + \ln[2\mathbf{Dob_3}]-1 \\
 &+  \ln[2(\mathbf{Dob_4}+ \mathbf{Dob_2})]-1\\
 \ln \lbrace K_5\rbrace_{n\geq 1}&= \ln{\mathbf{Bell_1}}+ \ln{\mathbf{Bell_2}} +  \ln{2} + \ln{\mathbf{Bell_2}} + \ln{2} + \ln{\mathbf{Bell_3}} \\
 &+  \ln{2} + \ln{\mathbf{Bell_4}}+ \ln{\mathbf{Bell_2}} \\
  \ln \lbrace K_5\rbrace_{n\geq 1}&= \ln{\mathbf{Bell_1}}+ 3\ln{\mathbf{Bell_2}} + \ln{\mathbf{Bell_3}} + \ln{\mathbf{Bell_4}}+3\ln{2}\\
  \ln \lbrace K_5\rbrace_{n\geq 1}&= 3\ln{2}+ \sum_{i=1}^{4} \mathbf{\Phi}_i \ln{\mathbf{Bell}_i}
\end{align*}
the proof follows immediatetly.
\end{proof}

\begin{theorem}
    In general, it is possible to compute the Kurepa sequence associated with logarithm as in equation \ref{entropy};
    \begin{align}
          \log_{base} \lbrace K_n\rbrace_{n\geq 1}=\mathcal{Q}+\sum_{i=1}^{n}\mathbf{\Phi}_i \log_{base}\mathbf{Bell}_i=\mathcal{Q}+\sum_{i=1}^{n}\mathbf{\Phi}_i\left[ \log_{base} \sum_{k\geq 1}^{n} S(n,k)\right]
    \end{align}   
\end{theorem}

\subsection{Logarithm of the shifted alternating Kurepa sequence}
According Bread \cite{beard1950coefficients}, $e^{e^x} \cdotp e^{-e^x}=I,$ and also from definitions \ref{id} and \ref{invBelo}, the complementary Bell number is given by 
$$\mathbf{invBell_n}=\mathbf{invDob_n}\cdotp{e} $$ and also we know that
$$e^{e^x} \cdotp e^{-e^x} = \mathbf{invDob_n}\cdotp\mathbf{Dob_n}=I$$ 
we can observe that $$\mathbf{invBell_n}\cdotp\mathbf{Bell_n}=\dfrac{1}{\mathbf{Bell_n}} \cdotp \mathbf{Bell_n}=I$$ 
where 
$$\dfrac{1}{\mathbf{Bell_n}}=\dfrac{e}{\mathbf{Dob_n}}$$
Now taking log base on both sides yield the following:
\begin{align}
\log_{base}\mathbf{|invBell_n|}&= \log_{base}\left[\mathbf{|invDob_n|}\cdotp e\right]\\ \nonumber
\log_{base}\mathbf{|invBell_n|}&= \log_{base}\left[\mathbf{|invDob_n|}\right] + \log_{base}(e)\\ \nonumber
    \log_{base}\dfrac{1}{\mathbf{Bell_n}}&= \log_{base}\left[\dfrac{e}{\mathbf{Dob_n}}\right]\\ \nonumber
   \log_{base}(1)-\log_{base}(\mathbf{Bell_n}) &=\log_{base}(e)-\log_{base}(\mathbf{Dob_n})
\end{align}
If we set $base=\exp{(1)}=e$ we have the natural Log \cite{beard1950coefficients}

\begin{align}\label{intropy}
    \ln\dfrac{1}{\mathbf{Bell_n}}&= \ln\left[\dfrac{e}{\mathbf{Dob_n}}\right]\\ \nonumber
   \ln(1)-\ln(\mathbf{Bell_n}) &=\ln(e)-\ln(\mathbf{Dob_n})\\ \nonumber
   -\ln(\mathbf{Bell_n}) &=1-\ln(\mathbf{Dob_n})\\\nonumber 
  \ln(\mathbf{|invBell_n|}) &=\ln\mathbf{|invDob_n|} + \ln(e) \\\nonumber
   -\ln(\mathbf{Bell_n}) &=\ln|\mathbf{invDob_n}| + 1 
\end{align}

\begin{theorem}
The natural log of the shifted alternating Kurepa sequence is
$$\ln\lbrace |A_n^s|\rbrace_{n\geq 1}=-\sum_{i=0}^{n}\ln\mathbf{Bell}_i +\mathcal{Q}$$
where $\mathcal{Q}=\ln(Constant).$

\end{theorem}

\begin{proof}
    \begin{align*}
   \lbrace A_n^s\rbrace_{n\geq 1}&=\sum_{i=1}^{n}A_i^s=A_1^s + A_2^s + A_3^s +A_4^s +\cdots + A_n^s\\
\lbrace A_n^s\rbrace_{n\geq 1}\cdotp e^{-x}&= A_1^s\cdotp e^{-x} + A_2^s \cdotp e^{-x}+ A_3^s \cdotp e^{-x}+A_4^s \cdotp e^{-x}+\cdots + A_n^s \cdotp e^{-x}\\ 
\mbox{if x=1 then:}\\
\lbrace A_n^s\rbrace_{n\geq 1}\cdotp e^{-1}&= (-1)^0 !1\cdotp e^{-1} + (-1)^1 !2\cdotp e^{-1} + (-1)^2 !3\cdotp e^{-1} \\
&+ (-1)^3 !4\cdotp e^{-1} +\cdots + A_n^s e^{-1}\\
&= 1\cdotp e^{-1} -  0\cdotp e^{-1} + 2\cdotp e^{-1} - 4\cdotp e^{-1} + 20\cdotp e^{-1} -100\cdotp e^{-1}+ 620\cdotp e^{-1}\\
&\quad -\cdots+(-1)^n(n-1)!e^{-1}
\end{align*}
We now compute for $n=5$,  we have  
\begin{align*}
    \lbrace A_n^s\rbrace_{n\geq 1}\cdotp e^{-1}&= (-1)^0 !1\cdotp e^{-1} + (-1)^1 !2\cdotp e^{-1} + (-1)^2 !3\cdotp e^{-1} + (-1)^3 !4\cdotp e^{-1} +(-1)^4 !5\cdotp e^{-1}\\
&= 1\cdotp e^{-1} -  0\cdotp e^{-1} + 2\cdotp e^{-1} - 4\cdotp e^{-1} + 20\cdotp e^{-1} -100\cdotp e^{-1}
\end{align*}
where 
\begin{align*}
A_i^s=&(-1)^n\mathbf{!i}=\sum_{i=1}^{n}(-1)^{i-1}(i-1)!= !0 - 1!+ !2 - 3! + 4!- 5!  + \cdots \\ \mbox{we can now compute}\\
 A_1^se^{-1}=(-1)^1\mathbf{!1}=\sum_{1}(-1)^{1-1}0!&= 1e^{-1}\\
 A_1^se^{-1}&=1e^{-1}=\mathbf{invDob_0}\\
 \ln{( |A_1^s|e^{-1})}&=\ln |\mathbf{invDob_0}|\\
  \ln{ |A_1^s|}-\ln{e}&=\ln |\mathbf{invDob_0}|\\
  \ln{| A_1^s|}&=\ln |\mathbf{invDob_0}|+\ln{e}=\ln |\mathbf{invBell_0}|\\
  \mbox{Similarly;}\\
\ln |A_2^s|&=\ln |\mathbf{invBell_2}|\\
\ln |A_3^s|&=\ln |\mathbf{invBell_3}|+ \ln2\\
\ln |A_4^s|&=\ln |\mathbf{invBell_5}| + \ln2\\
\ln |A_5^s|&=\ln |\mathbf{invBell_4}| + \ln20
\end{align*}
Now we take natural log of the shifted alternating Kurepa sequence
\begin{align*}
    \ln (\lbrace A_5^s\rbrace_{n\geq 1}\cdotp e^{-5})&= \ln(A_1^s\cdotp e^{-1}) + \ln(|A_2^s \cdotp e^{-1}) + \ln(A_3^s \cdotp e^{-1})\\
    &+ \ln(A_4^s \cdotp e^{-1})+ \ln(A_5^s \cdotp e^{-1})\\
\ln\lbrace A_n^s\rbrace_{n\geq 1}-\ln{e^{5}}&= \ln (1\cdotp e^{-1}) -  \ln (0\cdotp e^{-1}) + \ln(2\cdotp e^{-1}) - \ln(4\cdotp e^{-1})\\
&+ \ln(20\cdotp e^{-1}) - \ln(100\cdotp e^{-1})\\
\ln\lbrace |A_5^s|\rbrace_{n\geq 1}-5\ln{e} &=\ln |\mathbf{invDob_0}|+ \ln |\mathbf{invDob_2}|+ \ln |2\mathbf{invDob_3}|\\
&+\ln |2\mathbf{invDob_5}|+\ln |20\mathbf{invDob_4}|\\
\ln\lbrace |A_n^s|\rbrace_{n\geq 1}-5 &=\ln |\mathbf{invDob_0}|+ \ln |\mathbf{invDob_2}|+ \ln |2\mathbf{invDob_3}|\\
&+\ln |2\mathbf{invDob_5}|+\ln |20\mathbf{invDob_4}|\\
\ln\lbrace |A_5^s|\rbrace_{n\geq 1} &=\ln |\mathbf{invDob_0}+1|+ \ln |\mathbf{invDob_2}+1|+ \ln |2\mathbf{invDob_3}+1|\\
&+\ln |2\mathbf{invDob_5}+1|+\ln |20\mathbf{invDob_4}+1|\\
\ln\lbrace |A_5^s|\rbrace_{n\geq 1} &=\ln (\mathbf{invBell_0})+ \ln |\mathbf{invBell_2}|+\ln |2\mathbf{invBell_3}|\\
&+ \ln |2\mathbf{invBell_5}| + \ln |20\mathbf{invBell_4}|\\
&=\ln |\mathbf{invBell_0}|+ \ln |\mathbf{invBell_2}|+\ln |\mathbf{invBell_3}|\\
&+ \ln |\mathbf{invBell_5}| + \ln (\mathbf{invBell_4})+ \ln2+\ln2+\ln20\\
&=\ln |\mathbf{invBell_0}|+ \ln |\mathbf{invBell_2}|+\ln |\mathbf{invBell_3}|\\
&+ \ln |\mathbf{invBell_5}| + \ln |\mathbf{invBell_4}|+ \ln80\\
&=\sum_{i=0}^{5}\ln|\mathbf{invBell}_i| + \ln80\\
&\mbox{from equation \ref{intropy} we have }\\
\ln\lbrace |A_5^s|\rbrace_{n\geq 1}&=-\ln (\mathbf{Bell_0})- \ln (\mathbf{Bell_2})-\ln (\mathbf{Bell_3})
- \ln (\mathbf{Bell_5})\\ &- \ln (\mathbf{Bell_4})+ \ln80\\
\ln\lbrace |A_5^s|\rbrace_{n\geq 1}&=\ln(80)-\ln (\mathbf{Bell_0})- \ln (\mathbf{Bell_2})-\ln (\mathbf{Bell_3})
- \ln \mathbf{Bell_5}\\ &- \ln (\mathbf{Bell_4})\\
\ln\lbrace |A_5^s|\rbrace_{n\geq 1}&=-\sum_{i=0}^{5}\ln\mathbf{Bell}_i +\ln(80)
\end{align*}
the proof follows immediately.
\end{proof}

\begin{theorem}
    In general $$\log_{base}\lbrace |A_n^s|\rbrace_{n\geq 1}=-\sum_{i=0}^{n}\log_{base}\mathbf{Bell}_i +\mathcal{Q}$$
where $\mathcal{Q}=\log_{base}(Constant).$
\end{theorem}
\begin{proof}
    The proof of this is straightforward.
\end{proof}

\begin{corollary}
For base $2, 10$, and $\exp{(1)}=e$, then the $\log_{base}\lbrace A_n^s\rbrace_{n\geq 1}$ and $\log_{base}\lbrace K_n\rbrace_{n\geq 1}$ is given as:
\begin{enumerate}

\item $$ \log_{10} \lbrace K_n\rbrace_{n\geq 1}=\mathcal{Q}+\sum_{i=1}^{n}\mathbf{\Phi}_i \log_{10}\mathbf{Bell}_i$$

\item $$ \log_{2} \lbrace K_n\rbrace_{n\geq 1}=\mathcal{Q}+\sum_{i=1}^{n}\mathbf{\Phi}_i \log_{2}\mathbf{Bell}_i$$

    \item $$\log_{2}\lbrace| A_n^s|\rbrace_{n\geq 1}=-\sum_{i=0}^{n}\log_{2}\mathbf{Bell}_i +\mathcal{Q}$$
    \item $$\log_{10}\lbrace |A_n^s|\rbrace_{n\geq 1}=-\sum_{i=0}^{n}\log_{10}\mathbf{Bell}_i +\mathcal{Q}$$
    
\end{enumerate}
\end{corollary}
\begin{proof}
    The proof of this is straightforward.
\end{proof}



\section{Occupation number, canonical ensemble and normal ordering}\label{Physics}
In this section we investigate into some physical applications of Kurepa sequence. The problem of normal ordering has algebraic connection according to Schwinger, Katriel and many others to the exponential series, Stirling numbers of the second kind and Bell numbers \cite{osti_4389568, katriel1974combinatorial, katriel2000bell, mehta2009sudarshan, PhysRevLett.18.752, osti_4310789, blasiak2005combinatorics, blasiak2007combinatorics, mansour2016commutation}. We shall extend this results to the Kurepa sequence to check the Kurepa normal ordering and as well as Kurepa anti-normal ordering. We also consider the problem of occupation number and the canonical ensemble, we end this section with some investigation into some algebraic properties of Fermi-Dirac statistics \cite{planck1900theory, greiner2012thermodynamics, srivatsan2006gentile, dai2004gentile, gentile1940itosservazioni}.

\subsection{Bose normal ordering anti-normal ordering}
Blasiak and Horzela \cite{blasiak2007combinatorics} presented a comprehensive combinatorial approach for addressing operator ordering issues, specifically applied to the normal ordering of the powers and exponential of the boson number operator. The problem's solution was expressed by Bell and Stirling numbers that enumerate set partitions. This approach elucidated the intrinsic connections between ordering issues and combinatorial entities, while also demonstrating the analytical foundation of Wick’s theorem.
Interpreting $a$ and $a^{\mathbb{+}}$ as operators that create and annihilate a particle in a system leads to the occupancy number representation.
Let consider the boson creation and annihilation operators $a$ and $a^{\mathbb{+}}$ satisfying the commutator relation
$$[a, a^{\mathbb{+}}]=1.$$
The number operator $\mathcal{N}$ determines how many particles are present in a system. For a Hilbert space $\mathcal{H}$ generated by the number states $|n\rangle$, where $n = 0, 1, 2,\ldots$ counts the number of particles. Specifically,
$\mathcal{N}=a a^{\mathbb{+}}$ satisfying the relation $[a, \mathcal{N}]=a$ and $[a^{\mathbb{+}}, \mathcal{N}]=-a^{\mathbb{+}}$. In a Fock space, the creations and annihilation operators may be realized as
\begin{align*}
    a|n\rangle= \sqrt{n}|n-1\rangle, \quad  a^{\mathbb{+}}|n\rangle= \sqrt{n+1}|n+1\rangle \quad \mbox{and} \quad \mathcal{N}|n\rangle=n|n\rangle.
\end{align*}
Several authors including, J. Katriel, Mansour, Blasiak, Vagas et al., Louisell, and many more \cite{vargas2006normal, penson2002combinatorics, louisell1973quantum} \cite{blasiak2005combinatorics,blasiak2007combinatorics, mansour2016commutation, MANSOUR2007214, katriel2000bell, blasiak2006dobinski} worked on the exponential series of the number operator and used this method to express normal order of a particle system such as boson.
We can express $\mathcal{N}^k=(a a^{\mathbb{+}})^k$ in a normal order, for $k=2, 3,$
and $4$ we have: 
\begin{align*}
    \mathcal{N}^2&=(a^{\mathbb{+}})^2 a^2+a^{\mathbb{+}}a;\\
    \mathcal{N}^3&=(a^{\mathbb{+}})^3 a^3+3(a^{\mathbb{+}})^2 a^2+a^{\mathbb{+}}a;\\
    \mathcal{N}^4&=(a^{\mathbb{+}})^4 a^4+6(a^{\mathbb{+}})^3 a^3+7(a^{\mathbb{+}})^2 a^2+a^{\mathbb{+}}a
\end{align*}
The coefficients of $(a a^{\mathbb{+}})^k$ are the numbers of Stirling numbers of the second kind, $S(n,k)$, that is, 
$$ \mathcal{N}=(a a^{\mathbb{+}})^k=\sum_{k = 1}^{n} S(n, k) (a^{\mathbb{+}})^k a^k$$

\begin{definition}\label{odt}\cite{osti_4389568, mansour2016commutation}
    Let $x\in \BC,$ then $$e^{x(a^{\mathbb{+}}a)}= e^{(e^x -1)(a^{\mathbb{+}}a)}= \mathbf{Bell_n}(a^{\mathbb{+}}a)$$ is the normal ordering.
\end{definition}

\begin{definition}\label{anti}\cite{mehta2009sudarshan, PhysRevLett.18.752, osti_4310789}
    Let $x\in \BC,$ then $$e^{x(aa^{\mathbb{+}})}= e^{(1-e^{-x})(a^{\mathbb{+}}a)}= (-1)^n\mathbf{invBell_n}(a^{\mathbb{+}}a)$$ is the antinormal ordering.
\end{definition}

\subsection{Kurepa normal ordering and antinormal ordering}
The Bell polynomial is given by:
\[
\mathbf{Bell_n}(y) =\sum_{k=1}^{n} S(n, k) y^k=\mathbf{Tchd_n}(y)
\]
with exponential generating function $\sum_{n=0}^{\infty}\mathbf{Bell_n}(y)\dfrac{y^n}{n!}=e^{y(e^x - 1)}$
and the complementary Bell polynomial is given by:
\[
\mathbf{invBell_n}(y) =\sum_{k=1}^{n} S(n, k) (-y)^k=\mathbf{invTchd_n}(y)
\]
with exponential generating function $\sum_{n=0}^{\infty}\mathbf{invBell_n}(y)\dfrac{y^n}{n!}=e^{y(1-e^x)}$
\begin{equation}\label{Kad}
    (-1)^n\mathbf{invBell_n}(y) =(-1)^n\sum_{k=1}^{n} S(n, k) (-y)^k
\end{equation}
with exponential generating function $(-1)^n\sum_{n=0}^{\infty}\mathbf{invBell_n}(y)\dfrac{y^n}{n!}=e^{y(1-e^{-x})}$\\
more on this equation \ref{Kad} is given by the author and some others in a paper to appear.\\

\begin{definition} 
    The Kurepa polynomial $K_n(x)$ is defined as follows:
    \begin{equation*}
K_{n}(x)=\sum_{m=0}^{n-1} m!x^m,
 \end{equation*}
 from Theorem \ref{Sum} we can list some few examples(see table \ref{table 11}):
\begin{align}\label{PF}
K_{0}(x) &= 1\\ \nonumber
K_{1}(x) &= 1 + x\\  \nonumber
K_{2}(x) &= 1 + x + 2x^2\\  \nonumber
K_{3}(x) &= 1 + x + 2x^2 + 6x^3\\  \nonumber
K_{4}(x) &= 1 + x + 2x^2 + 6x^3 + 24x^4\\  \nonumber
K_{5}(x) &= 1 + x + 2x^2 + 6x^3 + 24x^4 + 120x^5\\  \nonumber
K_{6}(x) &= 1 + x + 2x^2 + 6x^3 + 24x^4 + 120x^5 + 720x^6\\  \nonumber
K_{7}(x) &= 1 + x + 2x^2 + 6x^3 + 24x^4 + 120x^5 + 720x^6 + 5040x^7
\end{align}
\

If $x=1$ one can obtain the Kurepa numbers $K_n$ is defined as follows
    \begin{equation*}
K_n=
\begin{cases}
	\text{0}    & \text{$\mathbb{F}_{0}$;}\\
	\text{$1$} & \text{$K_1$ see Table \ref{table 9}};\\ 
	\text{$K_{n\geq 2}=\mathbb{F}_{n\geq1}$}   & \text{for all positive integer $n$}.
\end{cases}
 \end{equation*} 
\end{definition}

    \begin{table}[!ht]\label{6T}
    \centering
    \begin{tabular}{|c|c|c|}\hline 
        $K_n(x)$ polynomial& $ \mathbb{F}_{n}(x)$ polynomial   \\\hline 
          $K_0(x) = N/A$ & $\mathbb{F}_0(x) = 0$ \\\hline \rowcolor{LightCyan}
        $K_1(x) = 1 $ & N/A  \\\hline
         $K_2(x) = 1 + x $ & $\mathbb{F}_1(x) = 1 + x$  \\\hline
          $K_3(x) = 1 + x + 2x^2$ & $\mathbb{F}_2(x) = 1 + x + 2x^2$  \\\hline
            $K_4(x) = 1 + x + 2x^2 + 6x^3$ & $\mathbb{F}_3(x) = 1 + x + 2x^2 + 6x^3$ \\\hline
              $K_5(x) = 1 + x + 2x^2 + 6x^3 + 24x^4$ & $\mathbb{F}_4(x) = 1 + x + 2x^2 + 6x^3 + 24x^4$   \\\hline
                $K_6(x) = 1 + x + 2x^2 + 6x^3 + 24x^4 + 120x^5$ & $\mathbb{F}_5(x) = 1 + x + 2x^2 + 6x^3 + 24x^4 + 120x^5$   \\\hline
    \end{tabular}
    \caption{Relations between Kurepa polynomial and $\mathbb{F}_n(x)$ polynomial \cite{OEIS:A003422} }
    \label{table 11}
\end{table}

\begin{theorem}
The Kurepa polynomial is the sum of Bell polynomials given by
$$\lbrace K_n\rbrace_{n\geq 1}(y)=\sum_{r=1}^{n} \Phi_r \mathbf{Bell_r}(y).$$    
\end{theorem}
\begin{proof}
    Using theorem \ref{Kprove} and also from tables \ref{table 11} and \ref{table 12}, it is easy to see that for $n=8$;
    \begin{align*}
     \lbrace K_8\rbrace_{n\geq 1}(y) &= \mathbf{Bell_1}(y) + 8 \mathbf{Bell_2}(y)  + 2 \mathbf{Bell_3}(y) + 56\mathbf{Bell_4}(y) + \mathbf{Bell_5}(y) + 4\mathbf{Bell_6}(y) 
\end{align*}
\end{proof}

\begin{theorem}
    Let $\mathbf{Bell_n}(a^{\mathbb{+}}a)=\sum_{k = 1}^{n} S(n, k) (a^{\mathbb{+}})^k a^k$ be the normal ordering of the boson number operator, then the Kurepa normal ordering(KOD) is given by 
    $$\lbrace K_n\rbrace_{n\geq 1}(a^{\mathbb{+}}a)=\sum_{r=1}^{n} \Phi_r \mathbf{Bell_r}(a^{\mathbb{+}}a) $$
   where $\Phi_r$ is coefficient of the  $\mathbf{Bell_n}(a^{\mathbb{+}}a).$
\end{theorem}
\begin{proof}
    For $n=4$, we have 
    \begin{align*}
        \lbrace K_4\rbrace_{n\geq 1}(a^{\mathbb{+}}a)&=\mathbf{Bell_1}(a^{\mathbb{+}}a) + \mathbf{Bell_2}(a^{\mathbb{+}}a)+2 \mathbf{Bell_2}(a^{\mathbb{+}}a)  + 2 (\mathbf{Bell_3}(a^{\mathbb{+}}a) + \mathbf{Bell_2}(a^{\mathbb{+}}a))    \\
        &= \mathbf{Bell_1}(a^{\mathbb{+}}a) + 5\mathbf{Bell_2}(a^{\mathbb{+}}a) + 2 \mathbf{Bell_3}(a^{\mathbb{+}}a)  \quad (\mbox{See table \ref{table 12}})      \\
        &=\sum_{k = 1}^{n} S(n, k) (a^{\mathbb{+}})a +5\sum_{k = 1}^{2} S(2, k) (a^{\mathbb{+}})^2 a^2 + 2\sum_{k = 1}^{3} S(3, k) (a^{\mathbb{+}})^3 a^3\\
        &=1 + a^{\mathbb{+}}a + 5((a^{\mathbb{+}})^2 a^2+a^{\mathbb{+}}a )+ 2((a^{\mathbb{+}})^3 a^3+3(a^{\mathbb{+}})^2 a^2+a^{\mathbb{+}}a )  \\
        &=\sum_{r=1}^{4} \Phi_r \mathbf{Bell_r}(a^{\mathbb{+}}a)
    \end{align*}

\begin{table}[!ht]
    \centering
    \begin{tabular}{|c|c|c|c|c|a|c|c|c|c|}\hline 
        Bell polynomials($\mathbf{Bell_n}(x)$)& $\mathbf{Bell}_n(a^{\mathbb{+}}a)$ polynomials  \\\hline 
         $\mathbf{Bell}_0(x) = 1$ &$\mathbf{Bell}_0(a^{\mathbb{+}}a) = 1$\\\hline
        $\mathbf{Bell}_1(x) = x$ & $\mathbf{Bell}_1(a^{\mathbb{+}}a) = 1 + a^{\mathbb{+}}a$ \\\hline
          $\mathbf{Bell}_2(x) = x+x^2$  & $\mathbf{Bell}_2(a^{\mathbb{+}}a)  = (a^{\mathbb{+}})^2 a^2+a^{\mathbb{+}}a$ \\\hline
            $\mathbf{Bell}_3(x) = x+3x^2+x^3$  & $\mathbf{Bell}_3(a^{\mathbb{+}}a) = (a^{\mathbb{+}})^3 a^3+3(a^{\mathbb{+}})^2 a^2+a^{\mathbb{+}}a$ \\\hline
              $\mathbf{Bell}_4(x) = x+7x^2+6x^3+x^4$ & $\mathbf{Bell}_4(a^{\mathbb{+}}a) = (a^{\mathbb{+}})^4 a^4+6(a^{\mathbb{+}})^3 a^3+7(a^{\mathbb{+}})^2 a^2+a^{\mathbb{+}}a$ \\\hline
    \end{tabular}
    \caption{Relations between Bell numbers and Bose normal ordering \cite{OEIS:A003422} }
    \label{table 12}
\end{table}
    
   from the table below we can easily from definition \ref{odt} the proof is immediate.
\end{proof}

\begin{theorem}\label{altk}
The shifted alternating Kurepa polynomial is the sum of complementary Bell polynomials \ref{Aliter} given by
$$\lbrace A_n^s\rbrace_{n\geq 1}(y)=\sum_{r=1}^{n} \Phi_r \mathbf{invBell_r}(y).$$    
\end{theorem}
\begin{proof}
    the proof is straightforward.
\end{proof}

\begin{corollary} From equation \ref{Kad} and Theorem \ref{altk}, it is easy to see that;
    $$(-1)^n\lbrace A_n^s\rbrace_{n\geq 1}(y)=(-1)^n\sum_{r=0}^{n} \Phi_r \mathbf{invBell_r}(y)$$   
\end{corollary}
\begin{proof}
    The proof of this is trivial using \ref{Kad} and Theorem \ref{altk}
\end{proof}

\begin{theorem}
    Let $\mathbf{invBell_n}(a^{\mathbb{+}}a)=\sum_{k = 1}^{n} (-1)^k S(n, k) (a^{\mathbb{+}})^k a^k$ be the anti-normal ordering of the boson number operator from definition \ref{anti}, then the Kurepa anti-normal ordering(KAD) is given by 
    $$\lbrace (-1)^n A_n^s\rbrace_{n\geq 1}(a^{\mathbb{+}}a)=(-1)^n\sum_{r=0}^{n} \Phi_r \mathbf{invBell_r}(a^{\mathbb{+}}a). $$
   where $\Phi_r$ is coefficient of $\mathbf{invBell_n}(a^{\mathbb{+}}a).$
\end{theorem}

\begin{proof}
    For $n=5$, we have 
    \begin{align*}
        \lbrace(-1)^n A_5^s\rbrace_{n\geq 1}(a^{\mathbb{+}}a)&=\mathbf{invBell_0}(a^{\mathbb{+}}a) + \mathbf{invBell_2}(a^{\mathbb{+}}a)-2 \mathbf{invBell_3}(a^{\mathbb{+}}a)  - 52 (\mathbf{invBell_5}(a^{\mathbb{+}}a) \\
        &+ 20\mathbf{invBell_4}(a^{\mathbb{+}}a))   \quad (\mbox{See table \ref{table 13}})  \\
        &=(-1)^n\sum_{r=0}^{5} \Phi_r \mathbf{invBell_r}(a^{\mathbb{+}}a)
    \end{align*}
Then from the table below, equation \ref{Kad} and definition \ref{anti} the proof is immediate. 
\begin{table}[!ht]
    \centering
    \begin{tabular}{|c|c|c|c|c|a|c|c|c|c|}\hline 
        Bell polynomials($\mathbf{invBell_n}(x)$)&$\mathbf{invBell}_n(a^{\mathbb{+}}a)$ polynomials  \\\hline 
         $\mathbf{invBell}_0(x) = 1$ &$\mathbf{Bell}_0(a^{\mathbb{+}}a) = 1$\\\hline
        $\mathbf{invBell}_1(x) = -x$ &$\mathbf{invBell}_1(a^{\mathbb{+}}a) =  -a^{\mathbb{+}}a$ \\\hline
          $\mathbf{invBell}_2(x) = -x+x^2$ & $\mathbf{invBell}_2(a^{\mathbb{+}}a)  = -a^{\mathbb{+}}a+(a^{\mathbb{+}})^2 a^2$ \\\hline
            $\mathbf{invBell}_3(x) = -x+3x^2-x^3$ & $\mathbf{invBell}_3(a^{\mathbb{+}}a) =-a^{\mathbb{+}}a+3(a^{\mathbb{+}})^2 a^2 -(a^{\mathbb{+}})^3 a^3$ \\\hline
              $\mathbf{invBell}_4(x) = -x+7x^2-6x^3+x^4$  & $\mathbf{invBell}_4(a^{\mathbb{+}}a) = -a^{\mathbb{+}}a+7(a^{\mathbb{+}})^2 a^2-6(a^{\mathbb{+}})^3 a^3+ (a^{\mathbb{+}})^4 a^4$ \\\hline
    \end{tabular}
    \caption{Relations between Bell numbers and Bose normal ordering \cite{OEIS:A003422} }
    \label{table 13}
\end{table}
\end{proof}



\subsection{Planck's distribution and Bell numbers}
The occupation number in statistical mechanics has been one of the fundamental problem in knowning the number of particles residing in the specific quantum state(energy level) and sum of numbers gives the total number of particles in the system \cite{zhou2018canonical, shinde2025large, dai2004gentile, gentile1940itosservazioni, katsura1970intermediate}. In this section we shall consider partition function and investigate its connection with Bell numbers by assuming $x=\beta\mathcal{E}_{states}. $

Let the total energy $\mathbb{E}_a=\sum_a n_a \varepsilon_a$ and $\mathcal{N}=\sum_a n_a$ be 
 a gas of $\mathcal{N}$ identical particles then the partition function, $Z$ is given by 
\begin{equation}\label{Part}
    Z=\sum_a \exp{(-\beta \mathcal{N} \sum_a \varepsilon_a)}=\sum_a \exp{(-\beta \mathbb{E}_a )}=\sum_a e^{(-\beta \mathcal{N} \mathcal{E})}
\end{equation}
where $a=1,2,3\ldots$ is the state, $\beta=\dfrac{1}{TK}$ with temperature $T$ and $\mathbb{E}_a=\sum_a \varepsilon_a$.
The mean number of particles is given by
\begin{equation}\label{mean}
    \bar{\mathbf{n}}_{gas}=-\dfrac{1}{\beta}\frac{\partial \ln{Z}}{\partial \varepsilon_{gas}},
\end{equation}
one can also express the Maxwell-Boltzmann distribution \cite{planck1901law} as
\begin{equation}
    \bar{\mathbf{n}}_{gas}=\dfrac{1}{\beta}\frac{\partial \ln{Z}}{\partial \varepsilon_{gas}}=\mathcal{N}\dfrac{e^{(-\beta \mathcal{E}_{i^{th}state})}}{\sum_{states} e^{(-\beta \mathcal{N} \mathcal{E}_{states})}}
\end{equation}
now it is possible to rewrite the partition function equation \ref{Part}, as a geometric series
\begin{align*}
    Z&=\left(\sum_{n_1=0}^{\infty} e^{(-\beta n_1 \mathcal{E}_1)}\right)\left(\sum_{n_2=0}^{\infty} e^{(-\beta n_2 \mathcal{E}_2)}\right)\left(\sum_{n_3=0}^{\infty} e^{(-\beta n_3 \mathcal{E}_3)}\right)\left(\sum_{n_4=0}^{\infty} e^{(-\beta n_4 \mathcal{E}_4)}\right)\cdots\\
    &=\left(\dfrac{1}{1- e^{-\beta \mathcal{E}_1}}\right)\left(\dfrac{1}{1- e^{-\beta \mathcal{E}_2}}\right)\left(\dfrac{1}{1- e^{-\beta \mathcal{E}_3}}\right)\cdots\\
    &=\prod_{states}\dfrac{1}{1- e^{-\beta \mathcal{E}_{states}}}
\end{align*}
taking natural $\log$ on both sides we have:
\begin{equation}
    \ln{Z}= -\sum_{states} \ln{(1-e^{-\beta \mathcal{E}_{states}})}
\end{equation}
substituting this into the mean number, equation \ref{mean} gives:

\begin{align}
    \bar{\mathbf{n}}_{gas}=-\dfrac{1}{\beta}\frac{\partial \ln{Z}}{\partial \varepsilon_{gas}}&=\dfrac{1}{\beta}\frac{\partial}{\partial \varepsilon_{gas}} \ln{(1-e^{-\beta \mathcal{E}_{states}})}\\\nonumber
    &= \dfrac{e^{-\beta \mathcal{E}_{states}}}{1- e^{-\beta \mathcal{E}_{states}}}=\dfrac{1}{e^{\beta \mathcal{E}_{states}}-1}
\end{align}
This distribution is called the Planck's distribution \cite{planck1901law, planck1900theory, shinde2025large, greiner2012thermodynamics}.

\begin{proposition}\label{planck}
Let $x=\beta\mathcal{E}_{states}$, then Planck's distribution $$ \bar{\mathbf{n}}_{gas}=\dfrac{1}{e^{\beta \mathcal{E}_{states}}-1}$$ can be written as
\begin{align*}
\bar{\mathbf{n}}_{gas}=\dfrac{\ln{e^1}}{\ln{e^{(e^{\beta\mathcal{E}_{state}}-1)}}}&=\dfrac{1}{e^{x}-1}=\dfrac{\ln{e^1}}{\ln{e^{(e^{x}-1)}}}\\
&=\dfrac{\ln{e^1}}{\ln{\frac{e^{e^x}}{e}}}=\dfrac{\ln{e^1}}{\ln(\mathbf{Dob_n})-1}      
\end{align*}
From definition \ref{Bell}, we can rewrite the distribution as 
\begin{align*}
    \bar{\mathbf{n}}_{gas}\sim\dfrac{\ln{e^1}}{\ln(\mathbf{Dob_n})-1}=\dfrac{\ln{e^1}}{\ln{\mathbf{Bell_n}}}.
\end{align*}
\end{proposition}

\begin{proof}
    The proof of this is straightforward from section \ref{Entrolog}.
\end{proof}

  \begin{theorem}\label{growth}
  From proposition \ref{planck}, the average mean number
      \begin{align*}
    \bar{\mathbf{n}}_{gas}\sim\dfrac{\ln{e^1}}{\ln{\mathbf{Bell}_n}}= \dfrac{1}{\ln{\mathbf{Bell}_n}},
\end{align*}
then
\begin{align*}
\ln\mathbf{Bell_n}&=\dfrac{1}{\bar{\mathbf{n}}_{gas}}\\
n(\ln n - \ln\ln n - 1)&\sim\dfrac{1}{\bar{\mathbf{n}}_{gas}}\\
\dfrac{1}{n(\ln n - \ln\ln n - 1)}&\sim \bar{\mathbf{n}}_{gas}
\end{align*}
where 
\begin{align*}
\dfrac{\ln \mathbf{Bell}_n}{n}&=\ln n-\ln\ln n-1+\dfrac{\ln\ln n}{\ln n}+\frac1{\ln n}\\
&+\dfrac12\left(\dfrac{\ln\ln n}{\ln n}\right)^2+O\left(\dfrac{\ln\ln n}{(\ln n)^2}\right).
\end{align*}
is the Brujin`s Bell growth bound \cite{lovasz2007combinatorial, de2014asymptotic}.
 \end{theorem}  

\begin{theorem}
In a bosonic system, if $\sum_{states}(\bar{\mathbf{n}}_{gas})=\mathcal{N}$ which is a constant, then $\sum_{states}(\bar{\mathbf{n}}_{gas})^{-1}$ is also a constant say $\mathcal{P}.$ The Kurepa sequence 
\begin{equation*}
    \lbrace K_n\rbrace_{n\geq 1}=\exp{\mathcal{P}} 
\end{equation*}
if the number of bosons is conserved. 
\end{theorem}

\begin{proof}
    From theorem \ref{growth},  $ \ln\mathbf{Bell_n}=(\bar{\mathbf{n}}_{gas})^{-1}$ and also from theorem \ref{lin}
$$ \ln \lbrace K_n\rbrace_{n\geq 1}=\mathcal{Q}+\sum_{i=1}^{n}\mathbf{\Phi}_i \ln\mathbf{Bell}_i=\mathcal{Q}+\sum_{i=1}^{n}\mathbf{\Phi}_i\left[ \ln \sum_{k\geq 1}^{n} S(n,k)\right].$$
The constants $\mathcal{Q}$ and $\mathbf{\Phi}_i $ can be absorbed in the $\ln\mathbf{Bell_n}$, thus we have 
\begin{align}
\sum_{i=1}^{n}  \ln\mathbf{Bell_i}=\sum_{states}(\bar{\mathbf{n}}_{gas})^{-1}= \ln \lbrace K_n\rbrace_{n\geq 1}.
\end{align}
Since a bosonic system must satisfy the condition $\sum_{states}(\bar{\mathbf{n}}_{gas})=\mathcal{N}$, \\where $\mathcal{N}$ is the total number of bosons in the system. We easily see that 
\begin{equation*}
    \ln \lbrace K_n\rbrace_{n\geq 1}=\mathcal{P},
\end{equation*}
and the Kurepa sequence 
\begin{equation}
    \lbrace K_n\rbrace_{n\geq 1}=\exp{\mathcal{P}} =e^{\mathcal{P}}
\end{equation}
this completes the proof.
\end{proof}

\subsection{Particle numbers( Fermi numbers)}
The exponential generating function $e^{(e^{x}+1)}$ yields an important phenomenon in elementary particle with spin half. We shall simply call these observations the Fermi numbers. Let
\begin{align}\label{eqnfermi}
     e^{(\exp{(x)}+1)}&= e \cdotp e^{\exp{(x)}}=\exp(1)\cdotp  \left(\sum_{k=0}^{\infty}\dfrac{k^n}{k!}\sum_{n=0}^{\infty}\dfrac{x^n}{n!} \right)\\ \nonumber
     &=\sum_{k=0}^{n}\mathbf{Fermi_n}\frac{x^k}{k!}=  e^{(\exp{(x)}+1)}
\end{align}
One immediately realizes the role Dobinski numbers play in these Fermi numbers, and the definition easily follows. 

\begin{definition}[Fermi numbers]
The Fermi numbers are given by 
\begin{align*}
     \mathbf{Fermi_{n}}&=e\sum_{k=0}^{\infty}\dfrac{k^n}{k!}= e\cdotp\mathbf{Dob_n }
\end{align*}
    for all nonnegative integer $n.$
\end{definition}
Also, 
\begin{align}
   \mathbf{Dob_n }=\dfrac{\mathbf{Fermi_{n}}}{\exp(1)}=\dfrac{\mathbf{Fermi_{n}}}{e}.
\end{align}
Below are few list of Fermi numbers:
\begin{align*}
    e^2&= e\sum_{n=1}\dfrac{k}{k!}=e\mathbf{Dob}_1= \mathbf{Fermi_{1}}\\ 
    2e^2&= e\sum_{n=2}\dfrac{k^2}{k!}=e\mathbf{Dob}_2= \mathbf{Fermi_{2}}\\ 
    5e^2&=e\sum_{n=3}\dfrac{k^3}{k!}=e\mathbf{Dob}_3= \mathbf{Fermi_{3}}\\ 
    15 e^2&= e\sum_{n=4}\dfrac{k^4}{k!}=e\mathbf{Dob}_4= \mathbf{Fermi_{4}}\\
    52e^2&= e\sum_{n=5}\dfrac{k^5}{k!}=e\mathbf{Dob}_5= \mathbf{Fermi_{5}}\\ 
    203e^2&= e\sum_{n=6}\dfrac{k^6}{k!}=e\mathbf{Dob}_6= \mathbf{Fermi_{6}}\\
    877e^2&= e\sum_{n=7}\dfrac{k^7}{k!}=e\mathbf{Dob}_7= \mathbf{Fermi_{7}} \\
     &\vdots \quad \quad \quad \quad \vdots \quad \quad \quad\vdots \quad \quad \quad \quad \quad \quad \quad\\
    e^2\mathbf{Bell_{n}}&=e\sum_{k=1}^{\infty}\dfrac{k^n}{k!}=e\mathbf{Dob}_{n}= \mathbf{Fermi_{n}}
\end{align*}

\begin{theorem}
The sum of Fermi numbers is the product of $(\exp{(1)})^2$ and the Kurepa sequence $ \lbrace K_n\rbrace_{n\geq 1}$, that is,
$$\sum_{r=0}^{n}\Phi_r\mathbf{Fermi_r}=e^2 \cdotp\lbrace K_n\rbrace_{n\geq 1}.$$
\end{theorem}
\begin{proof}
We proof for $n=8$
\begin{align*} 
  \lbrace K_n\rbrace_{n\geq 1}\cdotp e^2 &= \mathbf{Dob_1}e + 8\mathbf{Dob_2}e + 2\mathbf{Dob_3}e +56\mathbf{Dob_4}e + \mathbf{Dob_5}e \\& \quad + 4\mathbf{Dob_6}e\\ \nonumber
       e^2\lbrace K_n\rbrace_{n\geq 1} & =  \mathbf{Fermi_1} + 8\mathbf{Fermi_2} + 2\mathbf{Fermi_3} +56\mathbf{Fermi_4} + \mathbf{Fermi_5}\\ & \quad + 4\mathbf{Fermi_6}\\
     \mbox{the results is straight forward.} 
\end{align*}
\end{proof}

\begin{theorem}
    The product of the Kurepa sequence $ e\lbrace K_n\rbrace_{n\geq 1}$ and the ordinary factorial numbers $r!$ is the sum of the product of the derangement numbers with the Dobinski numbers $\mathbf{Der_n \cdotp Fermi_n}$, that is, 
    $$ e\lbrace K_n\rbrace_{n\geq 1} \cdotp r! = \sum_{k=1}^{n}\Phi_k (\mathbf{Der_k \cdotp Fermi_k})$$
\end{theorem}
\begin{proof}Let for $n=8$ and  $$e\lbrace K_8\rbrace_{n\geq 1} = \dfrac{\mathbf{Fermi_1} }{e} + 8\dfrac{\mathbf{Fermi_2}}{e} +2\dfrac{\mathbf{Fermi_3}}{e}+ 56\dfrac{\mathbf{Fermi_4}}{e}+ \dfrac{\mathbf{Fermi_5}}{e}+ 4\dfrac{\mathbf{Fermi_6}}{e}$$
\begin{align*}
      \textbf{Kurepa sequence} \cdotp r!e&=  e\lbrace K_n\rbrace_{n\geq 1} \cdotp r! \\
      &=r! \left(  \dfrac{\mathbf{Fermi_1} }{e} + 8\dfrac{\mathbf{Fermi_2}}{e} +2\dfrac{\mathbf{Fermi_3}}{e}+ 56\dfrac{\mathbf{Fermi_4}}{e}+ \dfrac{\mathbf{Fermi_5}}{e} \right. \\
     &\left.  \quad + 4\dfrac{\mathbf{Fermi_6}}{e}\right) \\
   &=  r!\dfrac{\mathbf{Fermi_1} }{e} + 8\cdotp r!\dfrac{\mathbf{Fermi_2}}{e} +2\cdotp r!\dfrac{\mathbf{Fermi_3}}{e}+ 56\cdotp r!\dfrac{\mathbf{Fermi_4}}{e}\\
   &+r!\dfrac{\mathbf{Fermi_5}}{e}+ 4\cdotp r!\dfrac{\mathbf{Fermi_6}}{e}\\
   &= r!e^{-1}\mathbf{Fermi_1} + 8(r!e^{-1})\mathbf{Fermi_2} + 2(r!e^{-1})\mathbf{Fermi_3} + 56(r!e^{-1})\mathbf{Fermi_4} \\
  & \quad + (r!e^{-1})\mathbf{Fermi_5}+ 4(r!e^{-1})\mathbf{Fermi_6}\\
   e\lbrace K_n\rbrace_{n\geq 1} \cdotp r! &= \mathbf{Der_1} \mathbf{Fermi_1} + 8\mathbf{Der_2} \mathbf{Fermi_2} + 2 \mathbf{Der_3} \mathbf{Fermi_3} +56 \mathbf{Der_4} \mathbf{Fermi_4} \\ 
  &\quad+ \mathbf{Der_5} \mathbf{Fermi_5}+ 4 \mathbf{Der_6} \mathbf{Fermi_6}
\end{align*}
\end{proof}

\begin{theorem}
    The product of ordinary factorial numbers $k!$ with the sum of Dobinski numbers $\mathbf{Dob_n}$ is the sum of the product of the derangement numbers with Dobinski numbers $\mathbf{Der_n \cdotp Fermi_n}$, that is,
    $$ k!\sum_{r=0}^n   \mathbf{Dob_r}=  \sum_{k=0}^n \Phi_k (\mathbf{Der_k \cdotp Fermi_k})$$
    where $\Phi_k>0$ is a constant.
\end{theorem}

\begin{proof}
From Theorem \ref{Sum} and Theorem \ref{ok}, we can see that,
\begin{align*}
k! \cdotp e\mathbf{Bell_n}&= k!(e^{-1})\mathbf{Dob_n}e\\
k! \mathbf{Dob_n}&= k! e^{-1}\mathbf{Fermi_n} = \mathbf{Der_n}\cdotp \mathbf{Fermi_n}\\
 k!\sum_{r=0}^n \mathbf{Dob_r} &=k!\left(e\mathbf{Bell_1}+ 8 e\mathbf{Bell_2}  + 2 e\mathbf{Bell_3} + 56\mathbf{Bell_4}+\cdots\right)\\
&= k!e^{-1}\mathbf{Fermi_1} + 8(k!e^{-1})\mathbf{Fermi_2} + 2(k!e^{-1})\mathbf{Fermi_3} + 56(k!e^{-1})\mathbf{Fermi_4}\\
& \quad +\cdots \\ 
&= \mathbf{Der_1} \mathbf{Fermi_1} + 8\mathbf{Der_2} \mathbf{Fermi_2} + 2 \mathbf{Der_3} \mathbf{Fermi_3} +56 \mathbf{Der_4} \mathbf{Fermi_4}+  \cdots
    \end{align*}
\end{proof}
\bigskip

Now we look at the complementary Fermi numbers(inverse Fermi numbers);
\begin{align}
     e^{-(\exp{(x)}+1)}&= \dfrac{e^{-\exp{(x)}}}{e}=\dfrac{1}{e}  \left(\sum_{k=0}^{\infty}(-1)^k \dfrac{k^n}{k!}\sum_{n=0}^{\infty}\dfrac{x^n}{n!} \right)\\ \nonumber
     &=\mathbf{invFermi_k}\exp(x)= \dfrac{\mathbf{invDob_k}}{e}\exp{(x)}
\end{align}

\begin{definition}[inverse Fermi numbers]
The inverse Fermi numbers are given by the series
\begin{align*}
     \mathbf{invFermi_{n}}&=\dfrac{1}{e}\sum_{n=1}^{\infty}(-1)^k \dfrac{k^n}{k!}= \dfrac{\mathbf{invDob_n }}{e}
\end{align*}
    for all for nonnegative integer $n$
\end{definition}
Also, 
\begin{align}
   \mathbf{invDob_n }=\mathbf{invFermi_{n}}\cdotp\exp(1)=\exp(1) \cdotp \mathbf{invFermi_{n}}
\end{align}

Below are some few examples for all $k\geq 0$;
\begin{align*}
    -1&=\sum_{n=1}(-1)^k\dfrac{k}{k!}=\mathbf{invFermi}_1\cdotp e\\ \nonumber
   0&=\sum_{n=2}(-1)^k\dfrac{k^2}{k!}=\mathbf{invFermi}_2\cdotp e\\ \nonumber
    1&=\sum_{n=3}(-1)^k\dfrac{k^3}{k!}=\mathbf{invFermi}_3\cdotp e\\ \nonumber
    1&=\sum_{n=4}(-1)^k\dfrac{k^4}{k!}=\mathbf{invFermi}_4\cdotp e\\ \nonumber
   -2&= \sum_{n=5}(-1)^k\dfrac{k^5}{k!}=\mathbf{invFermi}_5\cdotp e\\ \nonumber
  -9&= \sum_{n=6}(-1)^k\dfrac{k^6}{k!}=\mathbf{invFermi}_6\cdotp e\\ \nonumber
   -9&= \sum_{n=7}(-1)^k\dfrac{k^7}{k!}=\mathbf{invFermi}_7\cdotp e \\ \nonumber
     &\vdots \quad \quad \quad \quad \vdots \quad \quad \quad\vdots \quad \quad \quad \quad \quad \quad \quad\\
    \mathbf{invDob_{n}}&=\sum_{n=1}^{\infty}(-1)^k\dfrac{k^n}{k!}=\mathbf{invFermi}_{n}\cdotp e
\end{align*}

\begin{theorem}
The sum of $\mathbf{invFermi_r}$ is given by
$$\sum_{r=0}^{n} \Phi_r \mathbf{invFermi_r}=\dfrac{\lbrace A_n^s\rbrace_{n\geq 1}}{ e^{2}}$$
\end{theorem}
\begin{proof} we shall prove this with just an example by considering $n=5,$
    \begin{align*}
\lbrace A_5^s\rbrace_{n\geq 1}\cdotp e^{-1}&=\mathbf{invDob_0} + \mathbf{invDob_2} + 2\cdotp \mathbf{invDob_3} +2\cdotp \mathbf{invDob_5} + 20\cdotp \mathbf{invDob_4} \\
&\quad + 50\cdotp \mathbf{invDob_{5}}\\
\lbrace A_5^s\rbrace_{n\geq 1}\cdotp e^{-2}&=\mathbf{invDob_0}\cdotp e^{-1} + \mathbf{invDob_2}\cdotp e^{-1} + 2\cdotp \mathbf{invDob_3}\cdotp e^{-1} +2\cdotp \mathbf{invDob_5}\cdotp e^{-1} \\
&+ 20\cdotp \mathbf{invDob_4}\cdotp e^{-1} 
\quad + 50\cdotp \mathbf{invDob_{5}}\cdotp e^{-1}\\
\lbrace A_5^s\rbrace_{n\geq 1}\cdotp e^{-2}&=\mathbf{invFermi_0} + \mathbf{invFermi_2} + 2\cdotp \mathbf{invFermi_3} +2\cdotp \mathbf{invFermi_5} + 20\cdotp \mathbf{invFermi_4}\\
&\quad   + 50\cdotp \mathbf{invFermi_{5}}e\\
\end{align*}
\end{proof}

\begin{lemma}\label{Gas}
    Let $\mathbf{Fermi_n}$ and $\mathbf{Bose_n}(\mathbf{Bell_n)}$ be exponential generating functions (see \ref{eqnfermi} and \ref{Bell}) given by:
    \begin{align*}
        \mathbf{Fermi_n} = e^{e^x + 1} \quad; \quad
        \mathbf{Bose_n}  = e^{e^x - 1},
    \end{align*}
    then for any ideal gas, there exists $\mathbf{Gas_n} =e^{e^x - \sigma_i}$, the exponential function of both Fermions and Bosons where \\
$\sigma_i  = \begin{cases}
+1 = Boson \\\\
-1 = Fermion
\end{cases}$  and  $x = \beta\varepsilon$, where $\beta = \frac{1}{TK}$.
\end{lemma}

\begin{proof}
Arthur Weldon \cite{weldon1992thermal} gave the distribution $n_i $ for the decay of a particle. Now we let
\begin{align}
n_i = \frac{1}{e^x - \sigma_i}, \quad \quad
\sigma_i = \begin{cases}
+1 = Boson \\\\
-1 = Fermion
\end{cases}
\end{align}
rewrite:\\
\begin{align*}
 n_i = \frac{\ln e^1}{\ln(e^{e^x}- \sigma_i)} =  \frac{\ln e^1}{\ln e^{(e^x - \sigma_i)}} = \frac{\ln e^1}{\ln (\mathbf{Gas_n})}; \quad \quad  x = \beta \varepsilon,\quad  \beta = \frac{1}{TK}\\
\end{align*}
Now we observe that the $$\sum_{n=0}^{\infty}\mathbf{Gas_n}\dfrac{x^n}{n!}:= e^{e^x - \sigma_i} $$
\begin{align*}
    \mathbf{Gas_n} &:= \dfrac{1}{e^{\sigma_i}}\sum_{k=0}^{\infty}(-1)^k\dfrac{k^n}{k!} = \frac{e^{e^x}}{e^{\sigma_i}} = \dfrac{\mathbf{Dob_n}}{e^{\sigma_i}}
\end{align*}
$\quad \quad \quad \quad \quad \quad \quad \quad \quad \quad \quad \quad \quad \quad \quad \quad \quad \quad \quad \quad \quad \quad for\quad  i = +1, \quad i = -1$\\
\begin{center}
\begin{tabular}{ c c c c c}
$\mathbf{Dob_0}$ &:=$ \mathbf{Gas_0} e^{\sigma_i}$ & $:= e^{\sigma_i}$ & = $e$ & $e^2$\\  
$\mathbf{Dob_1}$ &:= $\mathbf{Gas_1}e^{\sigma_i}$ & $:= 1e^{\sigma_i}$ & = $1e$ & $1e^2$\\ 
$\mathbf{Dob_2}$ &:= $\mathbf{Gas_2} e^{\sigma_i}$ & := $2e^{\sigma_i}$ & = $2e$ & $2e^2$\\ 
$\mathbf{Dob_3}$ & := $\mathbf{Gas_3} e^{\sigma_i}$ & := $5e^{\sigma_i}$ & = $5e$ & $5e^2$\\
\vdots &&\vdots&&\\
$\mathbf{Dob_n} $&:= $\mathbf{Gas_n} e^{\sigma_i}$ & $:= \mathbf{Bell_n}e^{\sigma_i}$ & $\mathbf{Dob_n}$ & $\mathbf{Fermi_n}$\\ 
\end{tabular}
\end{center}
\end{proof}

\begin{theorem}
    For any particle in an ideal gas state, the decay of particles obeys the following distribution
    \begin{align*}
    n_i = \frac{\ln e^1}{\ln (\mathbf{Gas_n})}
    \end{align*}
     where $x = \beta \varepsilon $ and $\mathbf{Gas_n} = e^{e^x-\sigma_i}$ as defined previously.
\end{theorem}
Just like the nature of the Gentile statistics \cite{gentile1940itosservazioni, dai2004gentile, katsura1970intermediate}, we generalize the kurepa sequence for the decay of gas in a Fermi-Dirac and Bose-Einstein statistics. 
\begin{theorem}\label{Gasn}
Let $n$ be a nonnegative integer, the
\begin{align*}
    \lbrace K_n \rbrace_{n\geq 1}^{Gas} & = \sum^n_{r=1} \Phi_r \mathbf{Gas_r}
    \end{align*}
where $\Phi_r$ is the coefficient of $\mathbf{Gas_r},$ with $\mathbf{Gas_r} = e^{e^{x-\sigma_i}},$ and  $K_n$ is  kurepa factorial.
\end{theorem}
\begin{proof}Let 
    \begin{align*}
\lbrace K_n \rbrace_{n\geq 1}\cdotp e^{\sigma_i}& = (K_1 + K_2 + K_3 + \cdots+ K_n )e^{\sigma_i}\\
\text{if} \quad n=5,\\
\lbrace K_5 \rbrace_{n\geq 1} e^{\sigma_i}& = (K_1e^{\sigma_i} + K_2e^{\sigma_i} + K_3e^{\sigma_i} + \cdots )\\
& = 1e^{\sigma_i} + 2 e^{\sigma_i}+ 4e^{\sigma_i} + 10e^{\sigma_i} + 34e^{\sigma_i}\\
& =\mathbf{Dob_1}  +  \mathbf{Dob_2} + 2 \mathbf{Dob_2} + 2 \mathbf{Dob_3} + 2 \mathbf{Dob_4}\\
\lbrace K_5 \rbrace_{n\geq 1} e^{\sigma_i}& = \mathbf{Dob_1} + 3 \mathbf{Dob_2} + 2 \mathbf{Dob_3} + 2 \mathbf{Dob_4}\\
\lbrace K_5 \rbrace_{n\geq 1} & = \frac{1}{e^{\sigma_i}}( \mathbf{Dob_1} + 3 \mathbf{Dob_2} + 2 \mathbf{Dob_3} + 2 \mathbf{Dob_4})\\
\lbrace K_5 \rbrace_{n\geq 1}^{Decay} & = \mathbf{Gas_1} + 3 \mathbf{Gas_2} + 2 \mathbf{Gas_3} + 2 \mathbf{Gas_4}.
\end{align*}
\end{proof}
\section{Conclusion}\label{Conclus}
This article demonstrates the relationship between the Kurepa factorial, the Dobinski numbers, Bell numbers, and several others. We demonstrated that the summation of Bell numbers constitutes a Kurepa sequence; moreover, we partitioned the shifted alternating Kurepa sequence into the summation of complementary Bell numbers. We also examined the natural logarithm of the Kurepa sequence as well as the shifted alternating Kurepa sequence. Ultimately, we extended the findings of the Kurepa Decomposition to the normal ordering of certain elementary particle operators. Additionally, as an application, in statistical mechanics, we investigated the relationship between the Kurepa decomposition and the occupation number problem in the context of Bose-Einstein and Fermi-Dirac distributions. As an open question, we conjecture whether, for every Kurepa factorial or the $\mathbb{F}_n$ sequence, the greatest common divisor $\mathcal{F}(a)$ between successive elements of the $\mathbb{F}_n$ sequence is bounded above by $2$. The results is known for $\mathcal{F}(0)$ and $\mathcal{F}(4)$, the problem is to find all the other $\mathcal{F}(a)$ for which conjecture \ref{ATTA} is holds.

\bmhead{Acknowledgements}
I extend my gratitude to Akira Yoshioka and A. Sako for their hospitality during my stay at TUS, and also, I express my sincere appreciation to the participants of the "Noncommutative Geometry Seminar Kagurazaka" for their attentiveness to my presentation and the ensuing discussions. Also, I would like to thank Naruhiko Aizawa, Khalef Yaddaden and Simon Rutard for their suggestions and discussions in the final stage of this work.

\bibliography{sn-bibliography}

\end{document}